\documentclass{article}
\usepackage[affil-it]{authblk}
\expandafter\let\csname equation*\endcsname\relax
\expandafter\let\csname endequation*\endcsname\relax

\usepackage[dvipsnames,table]{xcolor}  % Coloured text etc.
\usepackage{natbib}
\setcitestyle{square,numbers}
\usepackage{hyperref}
\newcommand\myshade{85}
\colorlet{mylinkcolor}{violet}
\colorlet{mycitecolor}{YellowOrange}
\colorlet{myurlcolor}{Aquamarine}
\hypersetup{
    colorlinks=true,
    linkcolor  = mylinkcolor!\myshade!black,
    citecolor  = mycitecolor!\myshade!black,
    urlcolor   = myurlcolor!\myshade!black,
}

\usepackage[shortlabels]{enumitem}
\usepackage{caption}
\usepackage{algorithm}
\usepackage{algpseudocode}

\usepackage{bm}
\usepackage{graphicx}
\usepackage{xargs}
\usepackage{subcaption}
\usepackage{multirow}
\usepackage{amsmath}
\usepackage{amssymb}
\usepackage{amsthm}
\usepackage{mathrsfs}
\usepackage{amsbsy}
\usepackage{xspace}
\usepackage{mathtools}
\usepackage{float}
\usepackage{nicefrac}
\usepackage{pifont}
\usepackage{booktabs}
\usepackage{thm-restate}
\usepackage[letterpaper, total={6in, 8in}]{geometry}
\usepackage[capitalise, noabbrev]{cleveref} 
\usepackage{comment}

\usepackage{tikz}
\usetikzlibrary{calc,positioning}

\newtheorem{lemma}{Lemma}

\newtheorem{remark}{Remark}
\newtheorem{proposition}{Proposition}
\newtheorem{theorem}{Theorem}

\renewcommand{\AA}{{ {\mathcal{A}}}} % active collisional pairs
\newcommand{\CC}{{ {\mathcal{C}}}} % Configuration
\newcommand{\GG}{{\mathcal{G}}} 
\newcommand{\MM}{{ {\mathcal{M}}}} % mobility matrix
\newcommand{\MMstk}{{\mathcal{M}}^\mathrm{stk}}
\newcommand{\MMstkop}{\mathcal{M}_\infty^\mathrm{stk}}
\newcommand{\PPop}{\mathcal{P}_\infty}
\newcommand{\PP}{\mathcal{P}}
\newcommand{\UU}{{\mathcal{U}}} % velocity
\newcommand{\FF}{{\mathcal{F}}} % Force
\newcommand{\DD}{{\mathcal{D}}} % Mapping to collision pairs
\newcommand{\R}{\mathbb{R}}
\newcommand{\mat}[1]{\mathbf{#1}}
\renewcommand{\vec}[1]{\boldsymbol{#1}} 
\newcommand{\A}{{ \mat{A}}} % LCP matrix
\newcommand{\x}{{\vec{x}}} % collisional force magnitude
\newcommand{\xk}[1][k]{{\vec{x}^{(#1)}}} % use like \xk for x^{(k)} or \xk[k+1] for x^{(k+1)} (i.e., optional argument, and default value is k)
\newcommand{\z}{\vec{z}}
\renewcommand{\u}{\vec{u}}
\renewcommand{\v}{\vec{v}}

\renewcommand{\b}{\vec{b}}
\renewcommand{\c}{\vec{c}} 
\newcommand{\C}{\mat{C}} 
\newcommand{\D}{\mat{D}} 
 
\newcommand{\alphavec}{\boldsymbol{\alpha}} 
\newcommand{\p}{\vec{p}} % Search direction
\newcommand{\q}{\vec{q}} % Unconstrained search direction
\newcommand{\BPQN}{Bi-PQN\xspace}
\newcommand{\MPQN}{Mono-PQN\xspace}

\newcommand{\V}{\mat{V}}
\newcommand{\w}{\vec{w}}
\newcommand{\s}{\vec{s}}
\newcommand{\y}{\vec{y}}
\renewcommand{\H}{\mat{H}}
\newcommand{\B}{\mat{B}}

\newcommand{\U}{\mat{U}}
\newcommand{\prox}{\operatorname{prox}}

\DeclareMathOperator*{\argmin}{argmin}

\DeclareMathOperator{\dom}{dom}

\newcommand{\Np}{M} %notation for number of particles
\newcommand{\Nd}{N} %notation for total dof 
\newcommand{\Nc}{n} %notation for number of collisional paris
\newcommand{\defeq}{\stackrel{\text{\tiny def}}{=}} 

\DeclarePairedDelimiter\ceil{\lceil}{\rceil}

\graphicspath{{fig/}}

\begin{document}
\title{A Bifidelity Proximal Quasi-Newton Method for Dense Rigid Body Suspension Collision Resolution }
\author[1]{Nicholas Rummel\thanks{These authors contributed equally to this work.}}
\author[1]{Tyler Jensen$^*$}
\author[1]{Stephen Becker} 
\author[1]{Eduardo Corona} 
\affil[1]{Department of Applied Mathematics, University of Colorado Boulder}
\date{\today}
\maketitle
\begin{abstract}
    Direct numerical simulation of dense rigid body suspensions poses significant computational challenges. A popular approach to resolve collisions necessitates solving a linear complementary problem (LCP) per time step. Each matrix vector product (MVP) inside the LCP requires solving an expensive partial differential equation. In this work, we show the LCP can be solved efficiently, often in only three to four MVPs. Specifically, we develop a custom monofidelity proximal quasi-Newton (\MPQN) method and a bi-fidelity variant (\BPQN). Our approach is validated through an application to representative systems of dense Stokesian Janus particles. Notably, in contact resolution our \MPQN and \BPQN achieve $\approx 1.5 \times$ and $> 2 \times$ speed up respectively against a competitive baseline, with the latter method displaying robust, problem-size-independent convergence. For our largest simulation involving $216$ particles, our \BPQN cut total simulation runtime to five days, as compared to the eight days required by the prior state-of-the-art method.
\end{abstract}
\section{Introduction} \label{sec:background}

Dense Stokesian particulate flows model a wide array of complex materials of interest to fundamental science and its applications. In such materials, the interplay of short and long-range interparticle forces gives rise to emergent phenomena at larger scales; this leads to effective rheological properties such as jamming \citep{TrappePrasadCipellettiEtAl2001Nature}, phase transitions \citep{AndersonLekkerkerker2002Nature}, and shear thickening \citep{wagnerShearThickeningColloidal2009}, with applications to materials design, e.g., woven Kevlar \citep{LeeWetzelWagner2003JMaterSci,PrasetyaAndokoSuprayitnoEtAl2024JMEST}. In the context of the study of biological systems, recent work on active matter has revitalized interest in self-assembly and collective behavior of suspensions with self-propelled microswimmers, driven rotators, and crosslinked microtubules \citep{SaintillanShelley2015ComplexFluidsinBiologicalSystems, NeedlemanDogic2017NatRevMater}. Direct numerical simulations (DNS) of dense particulate suspensions are indispensable to the study of these materials; DNS serve as a critical, high-fidelity bridge between microscopic particle physics and macroscopic flow. They help overcome challenges in experimental observations, and this has motivated the recent development of fast, scalable and robust frameworks for large-scale DNS.

Various mathematical formulations have been proposed to capture interparticle forces, namely, Rotne–Prager–Yamakawa \citep{WajnrybMizerskiZukEtAl2013JFluidMech}, Stokesian dynamics \citep{BradyBossis1988AnnuRevFluidMech}, and multipole methods \citep{CichockiFelderhofHinsenEtAl1994JChemPhys, ClaeysBrady1993JFluidMech, SierouBrady2001JFluidMech}. Schemes formulating long-range forces obeying PDEs (Stokes, Laplace, Yukawa) include fictitious domain methods, immersed boundary methods, and boundary integral methods (BIMs) \citep{Maxey2017AnnuRevFluidMech}. Some models will also require incorporation of thermal fluctuations modeled by Brownian motion. Regardless of their formulation, large-scale high-fidelity DNS pose significant computational demands heightened by the wide range of spatial scales involved as well as the long time horizon inherent in slow relaxation of soft materials. 

This paper contributes substantial acceleration to address a computational bottleneck shared by an increasing number of candidate DNS frameworks for rigid body particle suspensions, namely, the cost of optimization-based collision resolution. Our key contributions are designed to be agnostic to which of the aforementioned mathematical formulations is employed. However, our implementation has been tailored to a boundary integral method (BIM). In the context of particulate Stokes flows, BIM has distinct advantages which arguably make it an ideal candidate, namely dimensionality reduction in the absence of body forces and well-conditioned linear systems \citep{KlintebergTornberg2016JournalofComputationalPhysics} arising from 2nd kind Fredholm equation formulations. All numerical experiments presented in this work use the BIM based Stokes mobility solver as discussed in \citep{CoronaVeerapaneni2018JournalofComputationalPhysics, YanCoronaMalhotraEtAl2020JournalofComputationalPhysics, KohlCoronaCheruvuEtAl2021arXiv210414068}.

\subsection{Computational bottleneck: Optimization-based Collision Resolution} \label{sec:background-collisions}
As we indicate above, regardless of what formulation of interparticle forces and numerical solvers are used, particularly if large-scale forces involve accurate solution of PDEs, incorporating collision resolution will result in a substantial increase to computational cost. As is argued in a number of recent works \citep{YanCoronaMalhotraEtAl2020JournalofComputationalPhysics, LuRahimianZorin2018arXiv181204719, LuRahimianZorin2017JournalofComputationalPhysics}, robust and practical DNS frameworks require addressing collisions regardless of whether lubrication theory would, in the limit, prevent them \citep{TophojMollerBrons2006PhysFluids, TophojMollerBrons2006PhysFluids,PalaniappanDaripa2002ZangewMathPhys}. Collision resolution approaches may be classified into two camps. Penalty methods prescribe collision potentials near particle boundaries, often equivalent to a sequence of progressively stiffer springs. This introduces considerable artificial stiffness, drastically limiting timestep size for explicit time-stepping schemes \citep{JanelaLefebvreMaury2005ESAIMProc}. This is especially problematic for dense suspension formulations where a PDE is solved at every timestep. Optimization-based methods provide an alternative that avoids artificial stiffness and can be effectively coupled with PDE-based formulations. By imposing non-overlapping constraints and solving the resulting linear complementary problem (LCP), collisions are resolved accurately. Solving the LCP, however, poses its own computational challenge: function, gradient and Hessian evaluations require expensive LCP matrix vector products (MVPs).

Conventional state-of-the-art LCP solvers include first order proximal gradient descent (PGD) methods, such as the accelerated Proximal Gradient Descent (A-PGD)\footnote{The reader may be familiar with this method under other titles such as FISTA or Nesterov PGD.} \citep{Nesterov1983ProcUSSRAcadSci, BeckTeboulle2009SIAMJImagingSci}, as well as active set methods \citep{Lemke1965ManagSci} and second order methods like interior point methods and minimum-map Newton \citep{Mangasarian1977JOptimTheoryAppl, Anitescu2006MathProgram, Niebe22015}. In the context of large-scale suspensions simulation, however, the observation was made that a heuristic variant of PGD that utilizes the Barzilai Borwein (BB-PGD) step size dominated other methods for moderate to large problems. In this context, BB-PGD owes its success to budgeting only \emph{one MVP per iteration}. In contrast, A-PGD may require more than one MVPs per iteration, and second-order methods require numerous MVPs to solve a linear system with a Krylov based solver at every iteration. From the vantage point that only hindsight can provide, it is clear that requiring many MVPs per iteration undermines the benefit provided by reducing the iteration count. This leads to the guiding principle of our work: 

\vspace{0.2 cm}
\centerline{\textbf{\emph{Fast LCP solvers should minimize total MVPs per iteration}.} }
\vspace{0.2 cm}
The methods presented in this work achieve the desired acceleration in two complementary ways. First, we develop a Proximal Quasi-Newton (\MPQN) method following the work of \citep{BeckerFadiliOchs2018arXiv180108691} with specializations particular to this application. Specifically, \MPQN first transforms the LCP into a constrained quadratic program, and then uses at most one gradient (hence one MVP) evaluation per iteration, while still incorporating curvature information. Second, we incorporate high and low \emph{fidelity} gradients in our Bifidelity Proximal Quasi-Newton (\BPQN) method. The low fidelity model is formed by coarsening the \emph{discretization} used to solve the so-called \emph{Stokes mobility problem} described in \Cref{sec:formulation}. The accuracy of evaluating the LCP matrix is controlled by the underlying discretization which solves the PDE. As described in \Cref{sec:multiFidelity}, careful selection of low and high fidelity MVPs and their incorporation into \BPQN further reduces the computational cost of the LCP solver.

For the LCP problem we study, well-tuned monofidelity algorithms converge quickly --- on the order of ten iterations --- because the second-kind Fredholm  formulations coupled with BIM provide LCPs that are exceptionally well conditioned \citep{KlintebergTornberg2016JournalofComputationalPhysics}. 
For this reason, we cannot use existing general purpose multifidelity optimization algorithms because these rely on sampling, which would immediately require more than ten MVPs.
These general methods also assume little to nothing about the structure of the lower fidelity models \citep{Fernandez-Godino2023ACSE, PeherstorferWillcoxGunzburger2018SIAMRev}. When designing our custom \BPQN algorithm we leveraged the  unique structure of the collision resolution problem. In particular, PDE-based particle simulations provide highly structured hierarchies of fidelity (e.g., hp-refinement), which our custom method will exploit. 

\subsection{Contributions}\label{sec:background-contri}
Our contributions can be summed up in two stages, corresponding to each of the two optimization solvers developed in this work: 

First, we present a custom implementation of a  \textbf{Proximal Quasi-Newton (\MPQN)} method, which consistently outperforms comparable first and second order methods. This method features careful integration of design choices specific to our problem. \MPQN uses quasi-Newton rank-$r$ updates taking advantage of a constant Hessian through a specialized dual problem, and importantly solving the dual problem does not necessitate additional expensive MVPs. Also, the optimal over-relaxation parameter is obtained through a closed form solution to a one dimensional constrained optimization problem. Finally, careful exploitation of the linear gradient allows for reuse of MVPs resulting in only one MVP per iteration.

Second, we develop a novel \textbf{Bifidelity Proximal Quasi-Newton (\BPQN)} algorithm. This method extends our \MPQN method to leverage readily available low-fidelity LCP matrices by incorporating them as initializations for the quasi-Newton Hessian approximation. Careful implementation of \BPQN minimizes the overall number of low and high-fidelity MVPs, leading to an improvement in computational efficiency.

We demonstrate that both of our methods result in a practical advantage. In simulations of Janus particles, our \MPQN and \BPQN achieve $\approx 1.5 \times$ and $> 2 \times$ speed up in contact resolution. Furthermore, our testing suggests the speed up achieved by \BPQN is problem-size-independent. For our largest simulation involving $216$ particles, our \BPQN cut total simulation runtime to $4.8$ days, as compared to the $7.8$ days required by the prior state-of-the-art method.

\subsection{Outline}\label{sec:background-outline}
\Cref{sec:formulation} introduces the mathematical formulation of the Stokes mobility problem. \Cref{sec:colRes} formulates how to resolve collisions via optimization and discusses alternative methods. \Cref{sec:pqn} describes our novel application and designed \MPQN method, and \cref{sec:multiFidelity} then extends the algorithm to the multifidelity case with \BPQN. \Cref{sec:res} demonstrates the viability of our approach with numerical experiments, and a discussion of conclusions and future work is in \cref{sec:disc}.

\section{Mathematical context: the Stokes mobility problem} \label{sec:formulation}

We start with a 
review of the broader mathematical context for our formulation of rigid body suspensions dynamics. When viscous forces dominate, it is reasonable to model the velocity of the surrounding fluid medium with the Stokes equations. A wide array of dynamical problems of interest can then be framed as solving a so-called \emph{Stokes mobility problem} at each time step; that is, given known total forces and torques applied to each particle, we must solve for the translational and rotational velocities. 

To this end, consider $\Np$ rigid particles suspended in a viscous Newtonian fluid that occupies a domain, taken to be $\mathbb{R}^3$ unless otherwise specified, exterior to the particles. Let $\{\vec{V}_j, \vec{\Gamma}_j, \c_j\}_{j=1}^{\Np}$ denote volume, boundary, and center of mass for each particle; the center of mass serves as a tracking point for translation. We denote the total force and torque applied to each particle as the pair $\{\vec{F}_{j},\vec{T}_{j}\}$, and the corresponding translational and rotational velocities as $\{\vec{U}_{j}, \vec{\Omega}_{j}\}$. We assume hydrodynamic forces and torques are a result of interparticle interactions communicated through the fluid medium, which can be integrated as moments of traction (surface force) on each particle's surface. Given fluid viscosity $\eta$, pressure $p_\text{fluid}$, and fluid velocity $\vec{u}$, one can compute the fluid stress  $\vec{\sigma} = -p_\text{fluid} \mat{I} + \eta \bigl[\nabla \vec{u} + (\nabla \vec{u})^\top\bigr]$. Assuming a near-zero Reynold's number regime (microscopic, viscous flow) and negligible body forces, it holds that $\nabla \cdot \vec{\sigma} = 0$ in the fluid domain. One obtains the Stokes equations from assuming the fluid is incompressible: 
\begin{subequations} \label{eq:stokes}
\begin{align}
        -\nabla p_\text{fluid}  + \eta \nabla^2 \vec{u}(\vec{z})  = 0 &, \; \forall \vec{z} \in \mathbb{R}^3 \setminus \cup_{j=1}^{\Np} \vec{V}_j\\
        \nabla \cdot \vec{u}(\vec{z})  = 0 &, \; \forall \vec{z} \in \mathbb{R}^3 \setminus \cup_{j=1}^{\Np} \vec{V}_j\\
        \vec{u}(\vec{z}) \rightarrow 0 &, \; \text{as } \vec{z} \rightarrow \infty .
\end{align}
\end{subequations}
Force-torque balance and a chosen model of slip (e.g., no-slip) impose boundary conditions on Eq.~\eqref{eq:stokes}:
\begin{subequations} \label{eq:stokes-bc}
\begin{align}
       \vec{u}(\vec{z}) =  \vec{U}_j(\vec{z}) +  \vec{\Omega}_j(\vec{z}) \times (\z - \c_j)&, \; \forall \vec{z} \in \vec{V}_j,\\
       \int_{\vec{\Gamma}_j} \vec{f} \, \mathrm{d}S =  \vec{F}_j,\\
       \int_{\vec{\Gamma}_j} (\z - \c_j) \times \vec{f} \, \mathrm{d}S =  \vec{T}_j
\end{align}
\end{subequations}
where $\vec{f} = \vec{\sigma}\cdot \vec{n}$ is the hydrodynamic traction (surface force) density on each particle and $\vec{n}$ is the outward facing surface normal. 

Given that the Stokes equations are linear, we note our modeling assumptions and Eqs.~\eqref{eq:stokes} and~\eqref{eq:stokes-bc} imply the map between input forces and torques and output translational and rotational velocities is itself linear. That is, if we define $\UU = \bigl(\cdots , U^x_j , U^y_j ,U^z_j , \Omega^x_j , \Omega^y_j ,\Omega^z_j,\cdots\bigr) \in \mathbb{R}^{6\Np}$, $\FF = \bigl(\cdots , F^x_j , F^y_j ,F^z_j , T^x_j , T^y_j ,T^z_j,\cdots\bigr)\in \mathbb{R}^{6\Np}$ vectors of rigid body velocities and forces and torques, respectively, there is a so-called \emph{Stokes mobility matrix} $\MM \in \mathbb{R}^{6\Np \times 6\Np}$ such that $\UU = \MM \FF$. This mobility matrix is generally dense and encodes interparticle interaction given our assumptions on the medium. Via an energy-dissipation argument, $\MM$ can be shown to be symmetric positive definite \citep{DurlofskyBradyBossis1987JFluidMech, KimKarrila1991}. Linearity also grants us, via the superposition principle, the ability to separate our problem into a sum of sub-problems corresponding to force-torque pairs coming from two distinct sources, namely, non-collisional and collisional; we will use subindices $nc$ and $c$ for associated quantities accordingly. We define $\vec{u}_{n c}$ as the solution of Eqs.~\eqref{eq:stokes} and \eqref{eq:stokes-bc} given non-collisional forces $\vec{F}_{j,n c}$ and torques $\vec{\Omega}_{j,n c}$; these typically stem from long-range interactions. $\vec{u}_{c}$ then solves Eqs.~\eqref{eq:stokes}~and~\eqref{eq:stokes-bc} given collisional forces $\vec{F}_{j,c}$ and torques $\vec{T}_{j,c}$, and via linearity, $\vec{u}=\vec{u}_{nc}+\vec{u}_{c}$. This allows us to solve the mobility problem in two stages: given our current particle configuration, we first compute corresponding non-collisional rigid-body velocities $\UU_{nc}$ via a Stokes PDE solve. By superposition, we have:  
\begin{equation}\label{eq:mobility}
    \UU = \MM (\FF_{n c} + \FF_c) = \UU_{nc} + \MM \FF_c;
\end{equation}
in the collision resolution stage, we must then find collisional forces and torques that ensure particles satisfy contact constraints, leading to the complementarity problem (CP). 
\begin{figure} [ht]
    \centering
    \begin{tikzpicture}[
        matrix_box/.style 2 args={
            draw, thick, fill=gray!5, 
            minimum height=#1, 
            minimum width=#2, 
            anchor=center
        },
        dim_line/.style={
            |-|, 
            thin, 
            gray, 
            shorten <=2pt, 
            shorten >=2pt
        },
        dim_text/.style={font=\small, text=black}
        ]
        %% M
        \node[matrix_box={1cm}{1cm}] (M) {};
        \node at (M) {$\MM$};
        % M Dimensions
        \draw[dim_line] ($(M.south west)-(0,0.3)$) -- ($(M.south east)-(0,0.3)$) node[midway, below, dim_text] {$6\Np$};
        \draw[dim_line] ($(M.south west)-(0.3,0)$) -- ($(M.north west)-(0.3,0)$) node[midway, left, dim_text] {$6\Np$}; 
        
        %% =
        \node[right=0.25cm of M, font=\Large] (eq) {$\approx$};
        
        %% P
        \node[matrix_box={1cm}{3cm}, right=1cm of eq] (P) {};
        \node at (P) {$\PP$};
        % P Dimensions
        \draw[dim_line] ($(P.south west)-(0,0.3)$) -- ($(P.south east)-(0,0.3)$) node[midway, below, dim_text] {$\Nd$};
        \draw[dim_line] ($(P.south west)-(0.3,0)$) -- ($(P.north west)-(0.3,0)$) node[midway, left, dim_text] {$6\Np$};
        
        %% M-hat 
        \node[matrix_box={3cm}{3cm}, right=1cm of P] (Mstk) {};
        \node at (Mstk) {$\MMstk$};
        % M-hat Dimensions
        \draw[dim_line] ($(Mstk.south west)-(0,0.3)$) -- ($(Mstk.south east)-(0,0.3)$) node[midway, below, dim_text] {$\Nd$};
        \draw[dim_line] ($(Mstk.south west)-(0.3,0)$) -- ($(Mstk.north west)-(0.3,0)$) node[midway, left, dim_text] {$\Nd$};
        
        %% PT
        \node[matrix_box={3cm}{1cm}, right=1cm of Mstk] (PT) {};
        \node at (PT) {$\PP^\top$};
        d
        % PT Dimensions
        \draw[dim_line] ($(PT.south west)-(0.3,0)$) -- ($(PT.north west)-(0.3,0)$) node[midway, left, dim_text] {$\Nd$};
        \draw[dim_line] ($(PT.south west)-(0,0.3)$) -- ($(PT.south east)-(0,0.3)$) node[midway, below, dim_text] {$6\Np$};
    \end{tikzpicture}
    \caption{The true mobility matrix \(\MM=
\PPop \MMstkop \PPop^*\) is approximated by the discretized operators \(\PP\MMstk\PP^\top\). It is prohibitively expensive to explicitly form \(\MMstk\), thus only MVPs are available.}
    \label{fig:mobilityMatrixDiag}
\end{figure}
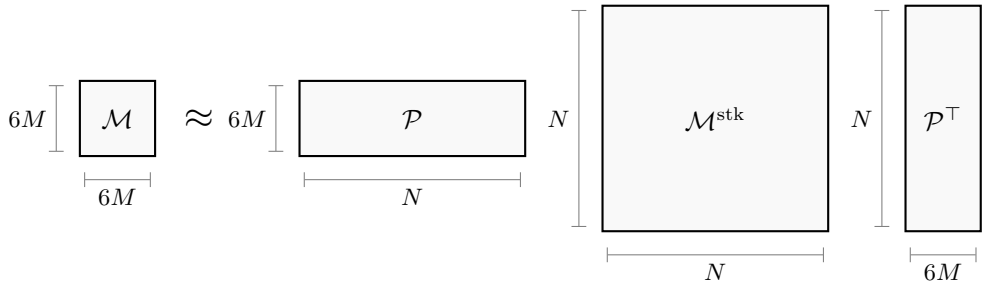
We note in practice, $\MM$ is seldom explicitly formed. Given the Stokes PDE formulation in Eqs.~\eqref{eq:stokes} and \eqref{eq:stokes-bc}, it is useful to conceptualize it as a composition of three operators 
$$\MM=
\PPop \MMstkop \PPop^* \approx 
\PP  \MMstk \ \PP^\top.$$ 
The linear operator $\PPop^*$ maps forces and torques to fluid tractions (i.e., $\vec{f} = \PPop^*[\FF]$), and its adjoint $\PPop$ maps fluid velocities to rigid-body velocities (i.e., $\UU = \PPop[\vec{u}]$). Similarly, the linear operator $\MMstkop$ maps tractions to fluid velocities (i.e., $\vec{u} = \MMstkop[\vec{f}]$). A given discretization method for the Stokes mobility problem corresponds to an approximate mobility matrix vector product (MVP), as represented graphically in \cref{fig:mobilityMatrixDiag}. If we represent $(\vec{u,f})$ with $\Nd$ degrees of freedom, then $\MMstkop$ and $\PPop$ are approximated by the dense matrix $\MMstk \in \mathbb{R}^{\Nd \times \Nd}$ and the block diagonal matrix $\PP \in \mathbb{R}^{6 \Np \times \Nd}$ respectively. $\PP^\top$ is the approximation of $\PPop^*$. 

This explains why $\MM$ is seldom formed: even for moderate number of particles, computing its entries accurately requires solving $6\Np$ large $\Nd \times \Nd$ linear systems. It also provides useful general information concerning MVPs; for most practical implementations of the PDE solve, preconditioned Krylov subspace methods such as CG or GMRES will be likely required, especially given the dynamic nature of our computational domain. Perhaps most importantly, this decomposition of $\MM$ and its MVP will allow us to analyze computational costs, conditioning and fidelity as they tie into design decisions such as the Stokes mobility formulation, its discretization and number of degrees of freedom $\Nd$. In this work, BIMs follow the formulation as described in \citep{CoronaGreengardRachhEtAl2017JournalofComputationalPhysics,CoronaVeerapaneni2018JournalofComputationalPhysics,YanCoronaMalhotraEtAl2020JournalofComputationalPhysics}.

\section{Collision Resolution}\label{sec:colRes}

The temporal evolution of the system of suspended particles can be represented by the center of mass of each particle $\c_j$ and the orientation of each particle's reference frame $\vec{\theta}_j$\footnote{The orientation can be represented by a unit quaternion instead of Euler angles. This is beneficial as it can alleviate numerical issues coming from the well-known problem of  \emph{gimbal lock}. When quaternions are used, the equation for angular velocity is modified to $\dot{\vec{\theta}}_j = \Psi \vec{\Omega}_j$ where $\Psi$ performs transformation from Euler angles to quaternions.}:
\begin{equation}\label{eq:timeEvolution}
    \begin{bmatrix} \dot{\c}_j\\ \dot{\vec{\theta}}_j \end{bmatrix} = \begin{bmatrix} \vec{U}_j\\
     \vec{\Omega}_j\end{bmatrix}.
\end{equation}
One can form the configuration $\CC = (\cdots,c^x_j,c^y_j,c^z_j,\theta^x_j,\theta^y_j,\theta^z_j,\cdots) \in \mathbb{R}^{6 \Np}$ to compactly represent all the particles. Its dynamics are governed by an overdamped equation of motion
\begin{equation}\label{eq:dynamSys}
    \dot{\CC}= \GG \UU_{n c}+\GG \MM \FF_c
\end{equation}
where  $\GG \in \mathbb{R}^{6\Np \times 6 \Np}$ 
is a block structured matrix with each block containing identity matrices that map the velocities $\U_j$ and $\vec{\Omega}_j$ to $\c_j$ and $\vec{\theta}_j$. 

We then wish to impose non-overlapping constraints to address collisions. Given that our particles are convex, collisions between any two particles happen at a single contact point; while there are ${\Np \choose 2}$ potential pairs, due to packing limits, far fewer (on the order of $\mathcal{O}(\Np)$) collisions are possible. 

Let $\vec{\Phi}: \mathbb{R}^{6 \Np} \rightarrow \mathbb{R}^{{\Np \choose 2}}$ map particle configuration to minimum separation distance between particle pairs, where $\Phi_\ell(\cdot)$ denotes the distance for the $\ell^{\mathrm{th}}$ pair. For practical purposes, we consider only a subset $\AA$ of potential collisions, excluding pairs for which collisions within the next timestep are impossible. That is, we define 
\begin{equation} \label{eq:collisional-set}
    \AA := \left\{ \ell \in \left\{1,\cdots,{\Np \choose 2}\right\} \middle\vert \Phi_\ell(\CC) \leq \delta_t \right\} \ \ , \ \ \Nc = \lvert \AA \rvert\
\end{equation}
where the collisional threshold $\delta_t$ depends on current particle velocities and timestep size $\Delta t$. We then introduce the vector $\x \in \mathbb{R}^{\Nc}$ of normal contact force magnitudes applied at corresponding contact points for pairs in $\AA$. Each non-penetration contact constraint can be enforced by requiring $x_{\ell}$ and $\Phi_{\ell}$ to be non-negative, and their product to be zero. This is denoted $0 \leq x_{\ell} \perp \Phi_{\ell} \geq 0$. The rigid body dynamics from our mobility problem along with elementwise complementarity constrains define the differential variational inequality (DVI) in the variables $\CC$ and $\x$:  
\begin{equation} \label{eq:DVI}
    \begin{gathered}
        \dot{\CC} = \GG \UU_{n c}+\GG \MM \FF_c \\
        \vec{0} \leq \vec{\Phi}_{\AA}(\CC) \perp \x \geq \vec{0}.
    \end{gathered}
\end{equation}

We then define $\DD \in \mathbb{R}^{6\Np \times \Nc}$ such that $\FF_c = \DD \x$; it is a sparse matrix, as each one of its columns $\vec{\D_\ell}$ registers total forces and torques resulting from the $\ell$th contact and has at most $12$ non-zero entries ($6$ for spheres). By specifying a time stepping rule, the configuration $\CC$ can be evolved through time. We use a forward Euler time stepping scheme, but others are possible:
\begin{equation} \label{eq:dvi}
    \begin{gathered}
        \CC^{(i+1)}=\CC^{(i)}+\Delta t \left(\GG \UU_{n c} +\GG \MM \DD \x \right) \\
        \vec{0} \leq \vec{\Phi_{\AA}}\left(\CC^{(i+1)}\right) \perp \x \geq \vec{0}.
    \end{gathered}
\end{equation}
Given the non-linearity of $\vec{\Phi}_{\AA}$, the complementarity condition in \Cref{eq:dvi} requires the solution of a \emph{nonlinear} complementarity problem (NCP)
\begin{equation} \label{eq:NCP}
    \vec{0} \leq \vec{\Phi_{\AA}}\left(\CC^{(i)}+\Delta t \left(\GG \UU_{n c} +\GG \MM \DD \x \right)\right) \perp  \x \geq \vec{0}.
\end{equation}
This nonlinearity poses both numerical and analytical challenges. Provided certain assumptions hold, one can linearize about the current configuration $\CC^{(i)}$ \citep{Anitescu2006MathProgram}, and instead solve the following \emph{linear} complementarity problem (LCP)
\begin{equation*} 
    \vec{0} \leq \vec{\Phi_{\AA}}\left(\CC^{(i)}\right) + \Delta t \nabla_{\CC} \vec{\Phi_{\AA}}(\CC^{(i)}) \left(\GG \UU_{n c} +\GG \MM \DD \x\right) \perp \x \geq \vec{0} . 
\end{equation*}
For rigid particles, it can be generally shown that  $\DD^\top=\left(\nabla_{\CC} \vec{\Phi}\right) \GG$ \citep{Anitescu2006MathProgram}, yielding the LCP
\begin{equation} \label{eq:LCP}
    \vec{0} \leq \A [\x] + \b \perp \x \geq \vec{0} 
\end{equation}
where 
\begin{equation} \label{eq:defAandb}
    \A = \DD^\top \MM \DD \in \mathbb{R}^{{\Nc} \times {\Nc}} \text{ and }
    \b = \tfrac{1}{\Delta t}\bigl(\vec{\Phi}\left(\CC^{(i)}\right) + \Delta t \DD^\top \UU_{n c}\bigr) \in \mathbb{R}^{\Nc}.
\end{equation}
Notice the LCP matrix $\A$ is symmetric positive semi-definite, making the LCP symmetric. 
In practice, the matrix $\A$ is in fact positive definite (and well-conditioned).

Since $\A$ is defined in terms of $\MM$, due to computational constraints, it is rarely formed as a full matrix. Applying $\A[\cdot]$ requires evaluating $\MM[\cdot]$ which already utilizes an iterative solver for a large linear system. Any methods which require solving a linear system at each iteration of the optimization method, i.e., applying $\A^{-1}[\cdot]$, would then involve nested linear iterative solvers and thus would be unlikely to stay within a \emph{total} budget of roughly a dozen MVPs. It is for this reason that we maintain that methods requiring one application of $\A[\cdot]$ per iteration are most appropriate. 

\subsection{The LCP as an Optimization Problem} \label{sec:colRes-opt}
A common reformulation of the LCP relies on its equivalence to the constrained quadratic program (CQP) with non-negativity constraints
\begin{equation} \label{eq:qp}
    \underset{{\x \geq \vec{0} }}{\operatorname{min}}\, \tfrac{1}{2} \x^\top \A[\x] + \x^\top\b. 
\end{equation}
This can be seen by deriving the KKT conditions and matching them with the above statements for the LCP. Second order optimization methods exist to solve the CQP \eqref{eq:qp}, but they effectively require linear solves for $\A$ or for a matrix composed from submatrices of $\A$. Hence, we restrict our attention to first order methods requiring only gradient information and, ideally, only  one MVP per iteration. 

\subsection{Optimization via a Local Model and Derivation of PGD} \label{sec:colRes-quadMdl}

Most optimization methods proceed by iteratively defining tractable local models; the presence of constraints can impose complications. In order to handle constraints, we consider an approach that splits the objective into two parts. For the moment, consider a generic optimization problem 
$$\min_{\x \in \mathbb{R}^\Nc} \; f(\x) + g(\x)$$
where the term $f$ is smooth and the term $g$ (which may encode a constraint) is not necessarily smooth. This can be solved by a \emph{forward-backward splitting} method where $f$ will be handled explicitly and $g$ in an implicit manner. Here we present one such method, and for reasons that will be clear shortly, it is referred to as \emph{proximal gradient descent} (PGD). For some $\tau^{(k)}>0$, define a local approximation $f^{(k)}$ of $f$ about the incumbent point $\xk$ as
$$f^{(k)}(\x) = f(\xk) + \nabla f(\xk)^\top (\x-\xk) + \frac{1}{2}\big\|\x-\xk\big\|_{\B_{\mathrm{PGD}}^{(k)}}^2 $$
where $\B_{\mathrm{PGD}}^{(k)} \defeq \tfrac{1}{\tau^{(k)}} \mat{I}$ and $\|\x\|_\B^2 \defeq \x^\top \B \x$ for any $\B\succ 0$. Then PGD iterates 
\begin{align*}
    \xk[k+1] &= \argmin_{\x \in \mathbb{R}^\Nc} \underbrace{f^{(k)}(\x) + g(\x)}_{m^{(k)}_{\mathrm{PGD}}(\x)} \\
    &= \argmin_{\x \in \mathbb{R}^\Nc} \frac{1}{2}\|\tilde{\x}^{(k)}- \x \|^2 + \tau^{(k)} g(\x) \quad\text{where}\quad \tilde{\x}^{(k)} = \xk - \tau^{(k)}  \nabla f(\xk).
\end{align*}
The second line follows by completing the square and multiplying by $\tau^{(k)}$.  Defining the \emph{proximal operator} of $g$ to be
\begin{equation}
    \label{eq:prox-pgd}
    \prox_{\tau g}(\y) \defeq  \argmin_{\x}\; \frac12\|\y-\x\|^2 + \tau g(\x),
\end{equation}
which is unique and well-defined whenever $g$ is convex, proper and lower semi-continuous, we can then write the generic form of proximal gradient descent as the iteration: 
\begin{algorithmic}[1]
    \State choose a step size $\tau^{(k)}$
    \State $\tilde{\x}^{(k)} = \xk - \tau^{(k)}  \nabla f(\xk)$
    \State $\xk[k+1] = \prox_{\tau^{(k)} g}( \tilde{\x}^{(k)} )$
\end{algorithmic}

\subsubsection{PGD for the CQP}\label{subsec:CQP}
Our constrained quadratic program in Eq.~\eqref{eq:qp} can be split as follows:
\begin{equation} \label{eq:fwdBwdSplt}
    \min_{\x \in \mathbb{R}^\Nc} \; \underbrace{\frac{1}{2} \x^\top \A[\x] + \x^\top\b}_f + \underbrace{\iota_{C}(\x)}_g, \quad \iota_{C}(\x) \defeq \begin{cases} 0, & \x\in C\\ \infty, & \x \not\in C\end{cases}
\end{equation}
where $\iota_{C}(\cdot)$ is the indicator function for the non-negative orthant.
For this choice of $f$, we have $\nabla f(\x) = \A[\x] + \b$, 
and the proximal operator for $\iota_{C}$ is the projection onto the non-negative orthant. 
This projection is easy to compute since it is done elementwise, taking $\xi \mapsto (\xi)_+ \defeq \max\{\xi,0\}$, which is a result of approximating the Hessian by a \emph{diagonal} matrix $\B_{\mathrm{PGD}}^{(k)} = \tfrac{1}{\tau^{(k)}} \mat{I}$.

\subsubsection{Step Size Selection}\label{sec:colRes-quadMdl-stepSize}
Forward-backward splitting methods begin by selecting an unconstrained search direction $\q^{(k)}$ and a corresponding step size $\tau^{(k)}$. For PGD, if $\nabla f$ is $L$-Lipschitz continuous, then it is well known that the algorithm converges if $\tau^{(k)} = \tau \in (0,\frac{2}{L})$, and there are also known results allowing for $\B^{(k)}$ to be a generic positive definite matrix (with appropriate bounds in the Loewner partial order sense)~\citep{Vu2012arXiv12102986}. 
In fact, much like the successive over-relaxation (SOR) technique for solving linear systems, forward-backward splitting methods also allow for an over-relaxation parameter $\eta$, as follows:
\begin{algorithmic}[1]
    \State choose an unconstrained step direction $\q^{(k)}$ 
    \State choose a step size $\tau^{(k)}$ 
    \State $\tilde{\x}^{(k)} = \xk - \tau^{(k)} \q^{(k)}$ 
    \State $\p^{(k)} \gets \prox_{\tau^{(k)} \iota_C}( \tilde{\x}^{(k)} ) - \xk$ \Comment{Define the over-relaxation search direction}
    \State choose over-relaxation parameter $\eta^{(k)}$
    \State $\xk[k+1] = \xk + \eta^{(k)} \p^{(k)}$
\end{algorithmic}
For certain choices, e.g., $\tau^{(k)} = \tau \in (0,\frac{1}{L})$ and $\eta^{(k)}=\eta \in (0,1]$, the proximal quasi-Newton method provably converges~\citep[Cor.~28.9]{BauschkeCombettes2017}. However, in practice the Lipschitz constant $L$ may not be known, or may be too conservative, and hence line search methods are used. Line search techniques can also be applied to the over-relaxation parameter 
$\eta^{(k)}$, which will be discussed in more detail in \cref{sec:pqn-stepSize}.
A visualization of one iteration of a forward-backward splitting method for the case $g=\iota_C$ is detailed in \cref{fig:stepSize}.

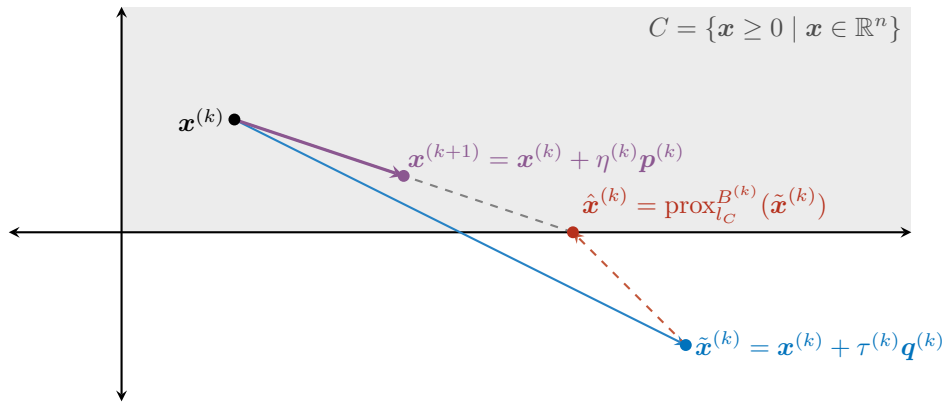
\begin{figure}[ht] 
    \centering
   
    \begin{tikzpicture}[>=stealth, scale=1.5]
        \fill[gray!15] (0,0) rectangle (7,2);
        \draw[<->, thick] (-1,0) -- (7,0) ; % x axis 
        \draw[<->, thick] (0,-1.5) -- (0,2) ; % y axis
        \node[below left, text=black!70] at (7,2) {$C=\left\{\x \geq 0 \mid \x \in \mathbb{R}^{\Nc}\right\}$};
    
        % 2. Define the coordinates for the points
        \coordinate (xk) at (1, 1);
        \coordinate (xtilde) at (5, -1);
        \coordinate (xhat) at (4, 0); 
        \coordinate (xnext) at (2.5, .5); 
    
        % unconstrained step
        \draw[->, RoyalBlue!80, thick] (xk) -- (xtilde);
        % proximal operator to obtain a feasible point
        \draw[->, BrickRed!80, dashed, thick] (xtilde) -- (xhat);
        % Over relaxation search direction 
        \draw[-, black!50, thick, dashed] (xk) -- (xhat);
        % Next iterate
        \draw[->, Purple!80!black, very thick] (xk) -- (xnext) ;
    
        % Draw the points and their formula labels
        \fill[black] (xk) circle (1.5pt) node[left] {$\xk$};
        \fill[RoyalBlue] (xtilde) circle (1.5pt) node[right] {$\tilde{\x}^{(k)}=\xk+\tau^{(k)}\vec{q}^{(k)}$};
        \fill[BrickRed] (xhat) circle (1.5pt) node[above right] {$\hat{\x}^{(k)}=\text{prox}_{l_{C}}^{B^{(k)}}(\tilde{\x}^{(k)})$};
        \fill[Purple!80!black] (xnext) circle (1.5pt) node[above right, xshift=-2pt, yshift=-2pt] {$\xk[k+1]=\xk+\eta^{(k)}\vec{p}^{(k)}$};
    
    \end{tikzpicture}
    \caption{An iteration of a forward-backward splitting algorithm: First, the unconstrained step $\tilde{\x}^{(k)} = \xk + \tau^{(k)} \q^{(k)}$ is taken where $\q^{(k)}$ is the unconstrained search direction shown in blue. Then, a feasible point is selected by applying the proximal operator $\hat{\x}^{(k+1)} = \prox_{\iota_C}^{\B^{(k)}}\bigl(\tilde{\x}^{(k)}\bigr)$ shown in red. Notice, that because the approximation to the Hessian is \emph{not necessarily} a diagonal matrix, the proximal operator may not apply the orthogonal projection onto the feasible set. Finally, an over-relaxation step is taken along the search direction $\p^{(k)}$ with parameter $\eta^{(k)}$ to obtain the next iterate $\xk[k+1]$ shown in purple.} \label{fig:stepSize}
\end{figure}

\subsection{Motivation for \MPQN} \label{sce:colRes-pqn}

Currently, variants of PGD (such as the Nesterov accelerated variants, cf.~\citep[Ch.~10.7]{Beck2017}) are the tool of choice when solving LCPs of moderate to large size in the context of collision resolution \citep{YanCoronaMalhotraEtAl2020JournalofComputationalPhysics}. PGD is a general tool and can be readily applied to Equation \eqref{eq:qp}, but it does not take advantage of the obvious structure of the optimization problem, namely a \emph{constant nondiagonal Hessian}. In the case of solving a QP without constraints, gradient descent converges linearly with a rate of $\tfrac{\kappa(\A) - 1}{\kappa(\A) + 1}=\mathcal{O}(1-\kappa(\A)^{-1})$ where $\kappa(\A)$ is the condition number of the matrix $\A$. This can be greatly improved upon by the use of quasi-Newton methods. Quasi-Newton methods approximate the Hessian using only information from gradient evaluations and the iterates themselves. This is done by creating low rank updates to an initial approximation that satisfy secant conditions. In the unconstrained case, quasi-Newton methods obtain a linear convergence rate of $\tfrac{\sqrt{\kappa(\A)} - 1}{\sqrt{\kappa(\A)} +1}$ when applied to QPs \citep{Atkinson1989}, as well as local superlinear convergence under various assumptions~\citep[Ch.\ 6]{NocedalWright2006}. Nesterov accelerated gradient descent can match the $\mathcal{O}(1-\kappa(\A)^{-1/2})$ rate but does not converge superlinearly.

In fact, on QPs, certain quasi-Newton methods with exact line searches reduce to the conjugate gradient (CG) method~\cite[Thm.\ 6.4]{NocedalWright2006}. The CG method is the optimal method among all gradient-based methods that produce estimates in the span of the initial point and subsequent gradients. The good performance of CG comes from choosing search directions that are conjugate to each other, so that the optimization is done in such a way such that the previous iterations' progress is not lost and there is no \textit{zig-zagging}. At each step the iterate is a global minimizer over the search directions explored thus far. 

Unfortunately, extending the good performance of quasi-Newton methods to \emph{constrained} optimization problems is difficult.
If $\xk[k+1] \gets \xk + \q_\text{Q-N}^{(k)} $ is the unconstrained quasi-Newton update, then $\xk[k+1]$ may be infeasible.  Defining 
$\xk[k+1] \gets \bigl( \xk + \q_\text{Q-N}^{(k)}\bigr)_+$ keeps $\xk[k+1]$ feasible, but drastically slows down convergence because of a subtle mismatch between metrics: taking intuition from CG, the conjugate directions do not align with the projection metric.  This can be fixed but requires care, especially in order to keep the iterations tractable. The faster convergence of quasi-Newton methods in the unconstrained case motivated the development of our \MPQN to solve \cref{eq:qp}.

\section{A Proximal Quasi-Newton Method} \label{sec:pqn}
This section describes our monofidelity proximal quasi-Newton (\MPQN) method and its specialization to solving the constrained quadratic program. The \MPQN method extends a forward-backward splitting framework similar to that discussed in \Cref{sec:colRes-quadMdl}. \MPQN optimizes a local quadratic model at each iteration $m_{\mathrm{PQN}}$ 
\begin{equation} \label{eq:pqnQuadMdl}
    \hat{\x}^{(k+1)} 
    = \underset{\x \in \mathbb{R}^\Nc}{\argmin} \underbrace{\nabla f(\xk)^\top (\x-\xk) + \tfrac{1}{2}\bigl( \x-\xk \bigr)^\top\B^{(k)}\bigl(\x-\xk \bigr) +\iota_C(\x)}_{m^{(k)}_{\mathrm{PQN}}(\x)}
\end{equation}
where $\B^{(k)}$ is a current approximation of the Hessian with a corresponding inverse $\bigl(\B^{(k)}\bigr)^{-1} = \H^{(k)}$. In contrast to PGD, here $\B^{(k)} = \B_0 + \U\U^\top - \V\V^\top$ where $\B_0 \in \mathbb{R}^{\Nc\times \Nc}$ and $\U, \V\in \mathbb{R}^{\Nc\times r}$ for $r\leq \Nc$. 
By completing the square, we can write the update as the solution to a \emph{weighted} proximal operator 
\begin{equation} \label{eq:prox-pqn}
 \hat{\x}^{(k+1)} 
    = \operatorname{prox}^{\B^{(k)}}_{\iota_C} (\tilde{\x}^{(k)}) \defeq \underset{\x \geq \vec{0} }{\operatorname{argmin}} \; \frac{1}{2} \|\tilde{\x}^{(k)}- \x \|_{\B^{(k)}}^2 
\end{equation}
where $\tilde{\x}^{(k)} = \xk + \H^{(k)}\bigl( \A [\xk] + \b\bigr)$.

Unlike Equation \eqref{eq:prox-pgd}, $\B^{(k)}$ is not diagonal, so its corresponding proximal operator is not separable and there is no closed form for $\operatorname{prox}^{\B^{(k)}}_{\iota_C} (\cdot)$. Instead, applying this weighted proximal operator requires solving a subproblem at each iteration. There are many approaches to address this. Schmidt et al.\ suggested solving the subproblem with a variant of PGD \citep{SchmidtBergFriedlanderEtAl2009ProcTwelfthIntConfArtifIntellStat}, and Kim et al.\ suggested leveraging a line-search on the original cost function \eqref{eq:qp} to guarantee the Armijo-Goldstein condition and using binding sets to modify $\B^{(k)}$ in such a way that stepping into the infeasible region happens infrequently \citep{KimSraDhillon2010SIAMJSciComput}. L-BFGS-B solves Equation \eqref{eq:prox-pqn} with an active set method \citep{ByrdLuNocedalEtAl1995SIAMJSciComput}.

We instead opt to pursue an approach that exploits the diagonal plus low-rank structure of $\B^{(k)}$, initially developed by \cite{BeckerFadili2012AdvNeuralInfProcessSyst} for the rank $r=1$ case, and later extend to general rank in \citep{BeckerFadiliOchs2018arXiv180108691}. These works were motivated by the update arising from a specific zero memory quasi-Newton method with the symmetric rank 1 (SR1) update $\B^{(k)} = \tfrac{1}{\tau^{(k)}} \mat{I} +\sigma \u\u^\top$ where  $\sigma \in \{-1,1\}$ \citep{BeckerFadiliOchs2018arXiv180108691}. In this case, a closed form solution to the weighted proximal operator exists. By introducing an auxiliary variable $\alpha = \u^\top \tilde{\x}^{(k)}$, there is a natural dual problem which leads to solving a one-dimensional root finding problem. This can be solved in $\mathcal{O}(\Nc\log(\Nc))$ time by locating the interval where the root occurs and then solving for the root in closed form within that interval. 

For our LCP, the zero memory variant ($r=1$) provided an unsubstantial speedup compared to PGD, possibly because it is not retaining enough curvature information. The work \citep{BeckerFadiliOchs2018arXiv180108691} theoretically studied the rank $r>1$ case but did not incorporate it into a quasi-Newton algorithm because of the increased cost of the sub-problem. Specifically, the weighted proximal operator can be found using the following proposition.
\begin{proposition}\label{prop-nnProx}
Given an real $\Nc\times \Nc$ positive definite matrix $\B = \D + \U\U^\top -\V\V^\top$ where $\D \succ 0$ is diagonal, $\U \in \mathbb{R}^{\Nc\times r}$,  $\V \in \mathbb{R}^{\Nc\times r}$ then 
\begin{equation} \label{eq:prop1-prox}
   \prox^\B_{\iota_{C}}(\tilde{\x}) = \bigl(\tilde{\x} + \C^{-1}\V\tilde{\alphavec} - \D^{-1}\U\alphavec\bigr)_+
\end{equation}
where $\begin{bmatrix} \alphavec^* & \tilde{\alphavec}^* \end{bmatrix}^\top \in \mathbb{R}^{2r}$ is the unique root of monotonic piecewise linear function
\begin{equation}\label{eq:prop1-rootFinding}
\mathcal{L}\left(\alphavec , \tilde{\alphavec}\right) = \begin{bmatrix}
        \U^\top \Bigl[\tilde{\x} + \C^{-1}\V\tilde{\alphavec} - \bigl(\tilde{\x} + \C^{-1}\V\tilde{\alphavec} - \D^{-1}\U\alphavec\bigr)_+\Bigr] \\
        \V^\top \Bigl[\tilde{\x} - \bigl(\tilde{\x} + \C^{-1}\V\tilde{\alphavec} - \D^{-1}\U\alphavec\bigr)_+\Bigr]
    \end{bmatrix}
\end{equation}
where $\C = \D + \U\U^\top$.
\end{proposition}

This is a direct application of the general convex analysis result given in Corollary 3.6 of \citep{BeckerFadiliOchs2018arXiv180108691}. More details of this result applied to the non-negative QP are provided in \cref{app:solveProx-deriv}. 

The theory directly applies to solving Equation \eqref{eq:prox-pqn}, but to our knowledge \MPQN is the first implementation of the method beyond the zero memory case. Because no analytic solution of the root of $\mathcal{L}(\vec{\alphavec})$ is available in general, the root of Equation \eqref{eq:prop1-rootFinding} is instead solved using a semi-smooth Newton method. While an exact solution is not available, the benefit of this approach is that the root finding problem is significantly smaller $2r \ll \Nc$ and requires no additional evaluations of $\A[\cdot]$, so the additional cost of the small semi-smooth Newton solve is negligible. The computational burden of the overall algorithm is still dominated by the MVPs. The rank parameter $r$ is bounded by the iteration count, so since the method takes a dozen or fewer steps, $r$ stays small. Furthermore, if the root is found to a high enough tolerance (which is a mild assumption since the objective is piecewise linear), then there are theoretical guarantees that inexact proximal operator evaluations still lead to a convergent algorithm \citep{LeeSunSaunders2014SIAMJOptim}. The full method is described in \cref{alg:pqn} and the details on how to quickly solve the weighted proximal operator are available in \cref{app:solveProx-ssnewton}. 

\begin{algorithm}[ht]
\caption{Monofidelity Proximal Quasi-Newton (\MPQN)} \label{alg:pqn}
\begin{algorithmic}
\Procedure{ProximalQuasiNewton}{$\A[\cdot], \b, \x^{(0)}; k_{\max}, \epsilon_{\mathrm{kkt}}$}
    \State $\H^{(0)} \gets \mat{I}$ 
    \For{$k \gets 0:k_{\max}-1$}
        \If{\(\langle \xk[k], \A[\xk[k]] + \b\rangle < \epsilon_{\mathrm{kkt}}\)}
            \State \Return $\xk[k]$
        \EndIf
        \State $\q^{(k)} \gets \xk -\H^{(k)}(\A[\xk] +\b)$
        \State $\tilde{\x}^{(k)} \gets \xk + \tau^{(k)} \q^{(k)}$
        \State $\hat{\x}^{(k)} \gets \prox^{\B^{(k)}}_{\x \geq \vec{0}}(\tilde{\x}^{(k)})$ \Comment{Via \Cref{alg:semiSmoothNewton} in \Cref{app:solveProx-ssnewton}}
        \State $\p^{(k)} = \hat{\x}^{(k)} - \xk$
        \State $\eta^{(k)} \gets$ StepSizeSelection($\A[\cdot], \b, \xk, \p^{(k)}$) \Comment{Via \Cref{alg:stepSize}}
        \State $\xk[k+1] \gets \xk + \eta^{(k)}\p^{(k)} $ 
        \State $\B^{(k)}, \H^{(k)} \gets$ quasiNewtonUpdate($\xk[k+1] - \xk$, $\eta^{(k)} \A[\p^{(k)}]$) \Comment{See, e.g., \cite[Ch.~6]{NocedalWright2006}}
    \EndFor
    \State \Return $\xk[k+1]$
\EndProcedure
\end{algorithmic}
\end{algorithm}

\begin{remark}[No Extra MVPs]
    Classical quasi-Newton updating rules provide fast MVPs of both the Hessian approximation $\B^{(k)}[\cdot]$ and its inverse $\H^{(k)}[\cdot]$. Both can be computed in $\mathcal{O}(\Nc r)$. As $\B^{(k)}$ is obtained from a formula that only requires the last \(r\) gradients \(\{\A[\x^{(i)}] + \b\}_{i=\ceil{k-r,1} }^k\) and the last \(r\) iterates \(\{\x^{(i)}\}_{i=\ceil{k-r,1}}^k\), no additional applications of $\A[\cdot]$ are required. 
\end{remark}

\subsection{\MPQN Step Size Selection}\label{sec:pqn-stepSize}

Empirically, \MPQN benefited considerably from choosing an optimal over-relaxation parameter
\begin{equation}\label{eq:stepSize-bwdOpt}
    \eta^{(k)} = \underset{\xk + \eta \p \geq \vec{0}}{\argmin} \; f(\xk + \eta \p)
\end{equation}
where the iteration index on $\p$ is dropped for ease of notation. This one dimensional optimization problem can be solved in closed form without using any additional MVPs since $f$ is quadratic. This is accomplished by first computing the optimal step length for the unconstrained case, given by the conjugate gradient step length
\begin{equation*} \label{eq:unconstrainedOptimalStepSize}
    \eta^\prime = -\frac{\p^\top (\A[\xk]+\b) }{\p^\top\A[\p]}.
\end{equation*}
Then, one can identify step sizes $\hat{\eta}_{\ell}$ where each constraint would be violated for a particular element of the minimizer $\xk_{\ell} + \hat{\eta}_{\ell}\p_{\ell}$. Because the step direction is selected to be a descent direction, if a particular $\hat{\eta}_{\ell}$ is negative then it can be disregarded. By taking the minimum over these candidate step sizes, the optimum $\eta^*$ can be found in closed form. This procedure is detailed in \cref{alg:stepSize}.

\begin{algorithm}[ht!] 
\caption{Optimal over-relaxation parameter}\label{alg:stepSize}
\begin{algorithmic}[1]
\Procedure{StepSizeSelection}{$\A[\cdot], \b, \xk, \p$}
    \State \(\eta^\prime =  -\dfrac{\p^\top (\A[\xk]+\b) }{\p^\top\A[\p]}\) \Comment{Unconstrained optimal step size} 
    \State \(
        \hat{\eta}_{\ell} = 
        \begin{cases} 
            - \frac{\xk_{\ell}}{p_{\ell}}, & p_{\ell} < 0 \\
            \infty, &\mathrm{otherwise}
        \end{cases}
    \)  \Comment{Identify step lengths where constraints are violated}
    \State \(\eta^* = \min(\eta^\prime, \hat{\eta}_0, \cdots, \hat{\eta}_{\Nc})\) \Comment{Constrained Optimal  Step Size} 
    \State \Return \(\eta^*\)
\EndProcedure
\end{algorithmic}
\end{algorithm}

\begin{remark}[Computing $\eta^*$ Requires No Extra MVPs]
    Notice, one can cache the vector $\A[\p^{(k)}]$. This allows the optimal step size $\eta^*$ to be computed without incurring an additional MVP. Then, the next objective function value and gradient computation can combine the precomputed quantities $\eta^*, \A[\p^{(k)}]$ and $\A[\xk]$. Explicitly, the value of the next MVP is computed as $\A[\xk[k+1]] = \A[\xk] + \eta^*\A[\p^{(k)}]$. Numerical errors will slowly accumulate so this is not stable for large $k$ but works well when $k$ is a dozen or less (or with  periodic recomputation).
\end{remark}

The step size procedure of \Cref{alg:stepSize} ensures global convergence of \Cref{alg:pqn} under mild assumptions.
\begin{restatable}{lemma}{stepSizeLemma}
    \label{lemma:convergence}
    Assume that for all \(k > 0\) there exists \(\mu > 0\) and \( \eta_{\text{min}}>0\) such that \(\B^{(k)} \succeq \mu\mat{I}\)  and that \(\eta^{(k)} > \eta_{\text{min}}\), then the iterates $\xk$ from \Cref{alg:pqn} with \(\tau^{(k)} = 1\) converge to the optimal solution of \Cref{eq:qp} when \Cref{eq:prox-pqn} is solved exactly.
\end{restatable}
\begin{proof}
This is a simple and direct application of Theorem 3.1 in \citep{LeeSunSaunders2014SIAMJOptim} specific to the quadratic case. The proof can be found in \Cref{app:stepSize}. 
\end{proof}
There are various other convergence results for proximal Newton-type methods in \cite{LeeSunSaunders2014SIAMJOptim}, such as linear and superlinear convergence under various assumptions. \Cref{alg:pqn} can be modified to follow these assumptions if one desires, but these assume eventual over-relaxation parameters of unity. In practice, such over-relaxation parameters were worse than the optimal step size found with \cref{alg:stepSize}. Furthermore, in our experiments \MPQN often converges in fewer than ten iterations $k \lesssim 10$, so asymptotic convergence rates are moot.

Interestingly, for PGD the optimal over-relaxation parameter $\eta^*$ degraded the performance compared to a fixed over-relaxation parameter $\eta =1$; this is not a contradiction since the optimal step size is optimal only in the greedy sense (of the current iteration) but may have eventual negative impacts for later iterations. We hypothesize that this observation is due to the angle between the forward and over-relaxation stepdirections $\arccos\left(\tfrac{(\q^\top \p)^2}{\|\q\|\|\p\|}\right)$ being more acute for \MPQN than for PGD. While \MPQN in constrained settings has no guarantees on the conjugacy of the new and previous search directions, we empirically observe that search directions are \textit{almost} conjugate. Thus doing local optimization over each individual over-relaxation parameter does not degrade global performance as it often does for PGD. 

\section{Multifidelity} \label{sec:multiFidelity}
Multifidelity methods are built on the assumption that the solutions of two related problems, one of which is significantly less expensive to compute (usually at a trade-off with fidelity), will be close to each other. Care is needed with such an assumption, as for any particular problem of interest there is no guarantee that a solution map will be robust to perturbation. Luckily, LCPs do fulfill this assumption; the solution map is locally Lipschitz. We use the notation $\mathbb{S}^{\Nc\times \Nc}_{++}$ to denote the set of $\Nc\times\Nc$ positive definite real-valued matrices. 
\begin{restatable}{lemma}{cottleLemma} \label{lemma:cottleLemma}
    Let $\A \in \mathbb{S}^{\Nc\times \Nc}_{++}$ and $\b \in \R^\Nc$ be such that the corresponding LCP has a solution $\x\in \mathbb{R}^\Nc$. Then, for any \(\delta \in (0, c(\A))\), there exists a neighborhood \(\mathcal{U}_\delta\) of \(\A\)
    \begin{equation*}
        \mathcal{U}_{\delta} = \{\B : \left\|\A-\B\right\|_{\infty} \leq \delta\}
    \end{equation*}
    such that for all \(\A^\prime\), \(\A^{\prime\prime} \in \mathcal{U}_\delta\) the following holds
    \[\left\|\x^\prime-\x^{\prime\prime}\right\|_{\infty} \leq \frac{\left\|(-\b)_+\right\|_{\infty}}{(c(\A) - \delta)^2}\left\|\A^\prime-\A^{\prime\prime}\right\|_{\infty}\]
    where $\x^\prime$ and $\x^{\prime\prime}$ are the unique solutions of $\left(\b, \A^\prime\right)$ and $\left(\b, \A^{\prime\prime}\right)$ and \(c(\A)\) is the \emph{fundamental quantity} of \(\A\)
\[c(\A)=\min _{\|\z\|_{\infty}=1}\left\{\max_{1 \leq i \leq n} \z_i(\A \z)_i\right\}.\]
\end{restatable}
This result is a specialization of Lemma 7.3.10 in \citep{CottlePangStone2009} assuming a fixed \(\b\). For more details on computing \(\c(\A)\) and application of \cref{lemma:cottleLemma}; see \cref{app:LipshitzConst}. \Cref{lemma:cottleLemma} states that perturbed LCPs still have unique solutions, and their solutions will be close to each other if the corresponding matrices are close to each other. Thus, LCPs are a strong candidate for multifidelity methods. We now focus on multifidelity methods in the context of the Stokes mobility problem.

\subsection{Multifidelity For the Stokes mobility problem}
The Stokes mobility problem can be solved to varying levels of accuracy, or \emph{fidelity}. Recall that the LCP matrix is of the form shown in \cref{fig:mobilityMatrixDiag} and \cref{eq:defAandb} 
\[\A = \DD^\top \PP \MMstk \PP^\top \DD.\]
The computational cost is dominated by the application of $\MMstk$ which involves a Stokes PDE solve; the exact composition of $\MMstk$ depends on which numerical method is used to formulate and solve the Stokes mobility problem and to what accuracy. There are multiple different notions of what could be considered fidelity in this context.
\begin{itemize}
    \item \textbf{Coarseness of discretization:} Each instance of \(\MMstk\) involves a discretization of the underlying mobility problem; the level of refinement can then be tied to fidelity. In our BIM approach, order-\(p\) spherical harmonics bases are used to represent densities on each surface, corresponding to \(N = 6 \Np p(p + 1)\) total degrees of freedom. Applying \(\MMstk\) has an asymptotic cost of \(\mathcal{O}(M(p^4+p^2))\) and can be improved to \(\mathcal{O}(M(p^3 \log p+p^2))\) applying FFT-based acceleration of near-evaluation \cite{CoronaVeerapaneni2018JournalofComputationalPhysics}. This polynomial dependence on $p$ implies that coarser discretizations can lead to substantially faster low fidelity MVPs.
    \item \textbf{Tolerances:} Many of the methods used to construct \(\MMstk\) require some type of iterative method, particularly when solving linear systems. These iterative methods solve a system up to a desired tolerance, which could be higher or lower depending on the desired fidelity. When constructing \(\MMstk\) for this work, a traction BIE solve is needed which is done with GMRES until the tolerance \(\epsilon_\mathrm{gmres}\) is met.
    \item \textbf{Sparsification:} Depending on how \(\MMstk\) is constructed, certain components could be purposely \emph{sparsified}. For example, this could be done by only considering self or near field interactions in the boundary integral kernels.
\end{itemize}

Once a low fidelity model is specified, a corresponding low fidelity LCP matrix $\hat{\A}$ can be defined. Although a low fidelity model will be less computationally expensive than the high fidelity model, this does not mean that the low fidelity model is so inexpensive that we can form \(\hat{\A}\) densely. Therefore, like the high fidelity \(\A\), we assume that we only have access to a low fidelity MVP \(\hat{\A}[\cdot]\). Furthermore, there is no guarantee that a given \(\hat{\A}\) will lead to a computational advantage in a multifidelity algorithm compared to a monofidelity algorithm. An appropriate balance of accuracy and computational efficiency of the low fidelity \(\hat{\A}\) is needed. Thus, finding a suitable \(\hat{\A}\) requires experimentation to tune the low fidelity model's parameters.

\subsubsection{Choosing Multifidelity MVPs for the BIM approach}
We consider low fidelity models parameterized by the coarseness of discretization dictated by spherical harmonic order \(p\) and linear BIE solver tolerance \(\epsilon_\mathrm{gmres}\); while sparsification was considered in our initial experiments, we found its performance to be less reliable. In order to identify which parameters to choose, we need to characterize computational cost and accuracy of differing low fidelity models on problems of interest. 

To this end, fifty example LCPs were generated from the simulation of spherical amphiphilic particles. For each example, dense representations of $\A$ were made for \(p \in \{3,4,5,6,7,8\}\) and \(\epsilon_{\mathrm{gmres}} = \{10^{-5}, 10^{-6}, 10^{-7}, 10^{-8}\}\) with the high fidelity \(\A\) defined by \(p = 8, \epsilon_{\mathrm{gmres}} = 10^{-8}\) with all other simulation details the same as in \cref{sec:res-offline}. To ensure a fair comparison, all low fidelity dense matrices were formed from particle configurations obtained during a simulation using a well tested monofidelity collision resolution algorithm. Then, for each pair $(p,\epsilon_\mathrm{gmres})$ and each particle configuration, we compute the average MVP time and absolute error between the dense low fidelity and the dense high fidelity matrices \(\hat{\A}\) and \(\A\); these dense LCP matrices were formed by applying the MVP to each column of the identity, requiring several days of computation time. 

Our aim then is to identify specific pairs (\(p\),\(\epsilon_\mathrm{gmres}\)) for which the gain in fidelity of \(\hat{\A}\) (measured as log of absolute error) relative to computational cost (in high fidelity MVPs) yields a locally optimal trade-off. Such parameter combinations are good low fidelity model candidates.

In \cref{fig:lofiMVPs}, the plot on the right shows a heatmap for the ratio between the average $\A$ MVP time and the average $\hat{\A}(p,\epsilon_\mathrm{gmres})$ MVP time, i.e., how many low fidelity MVPs can be computed for the cost of one high fidelity MVP. The plot on the left is the corresponding heatmap for the log of the absolute error between low fidelity and high fidelity dense matrices with respect to the infinity norm.

\begin{figure}[ht!]
     \centering 
    \includegraphics[width=0.45\textwidth]{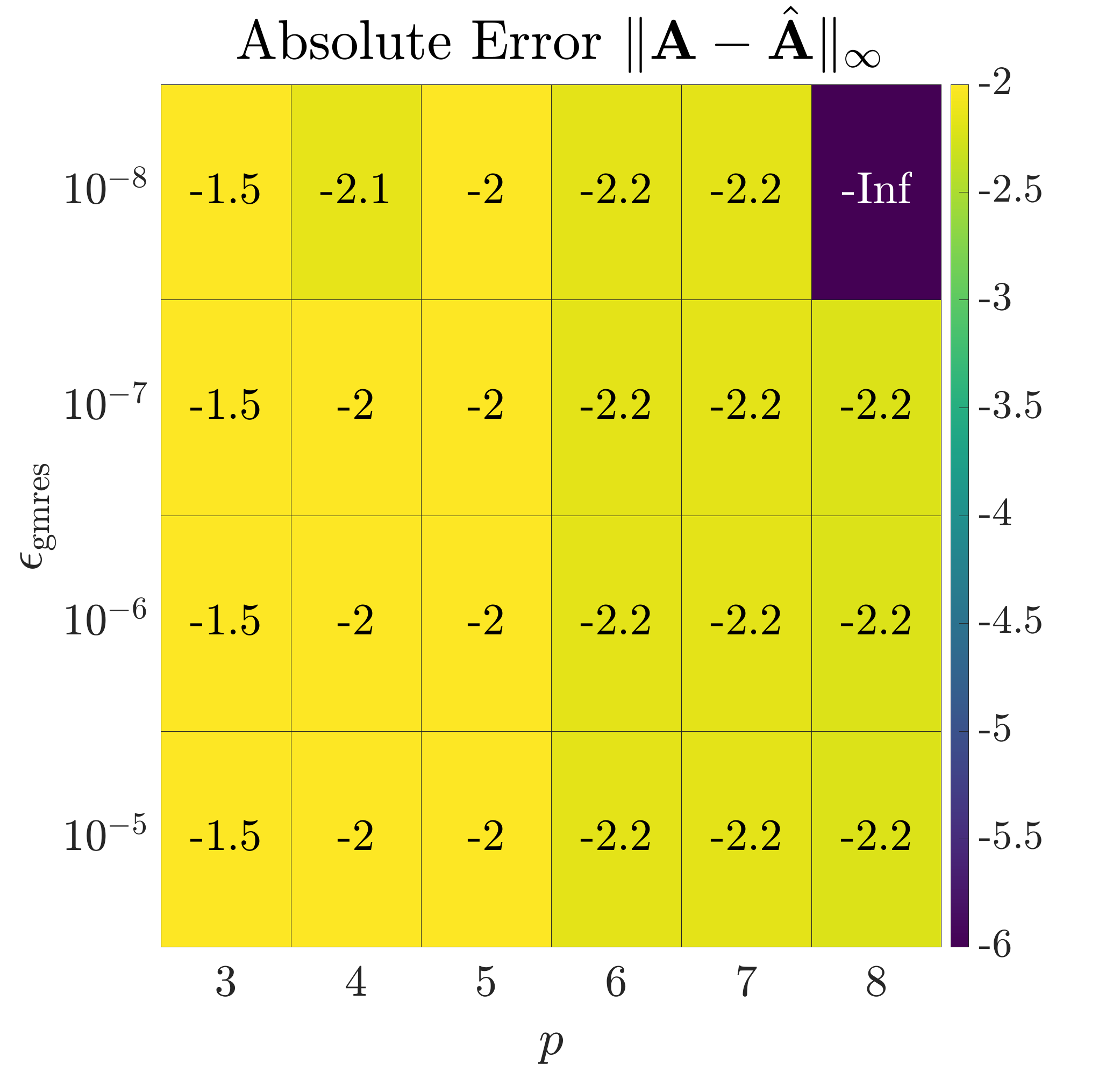}
    \includegraphics[width=0.45\textwidth]{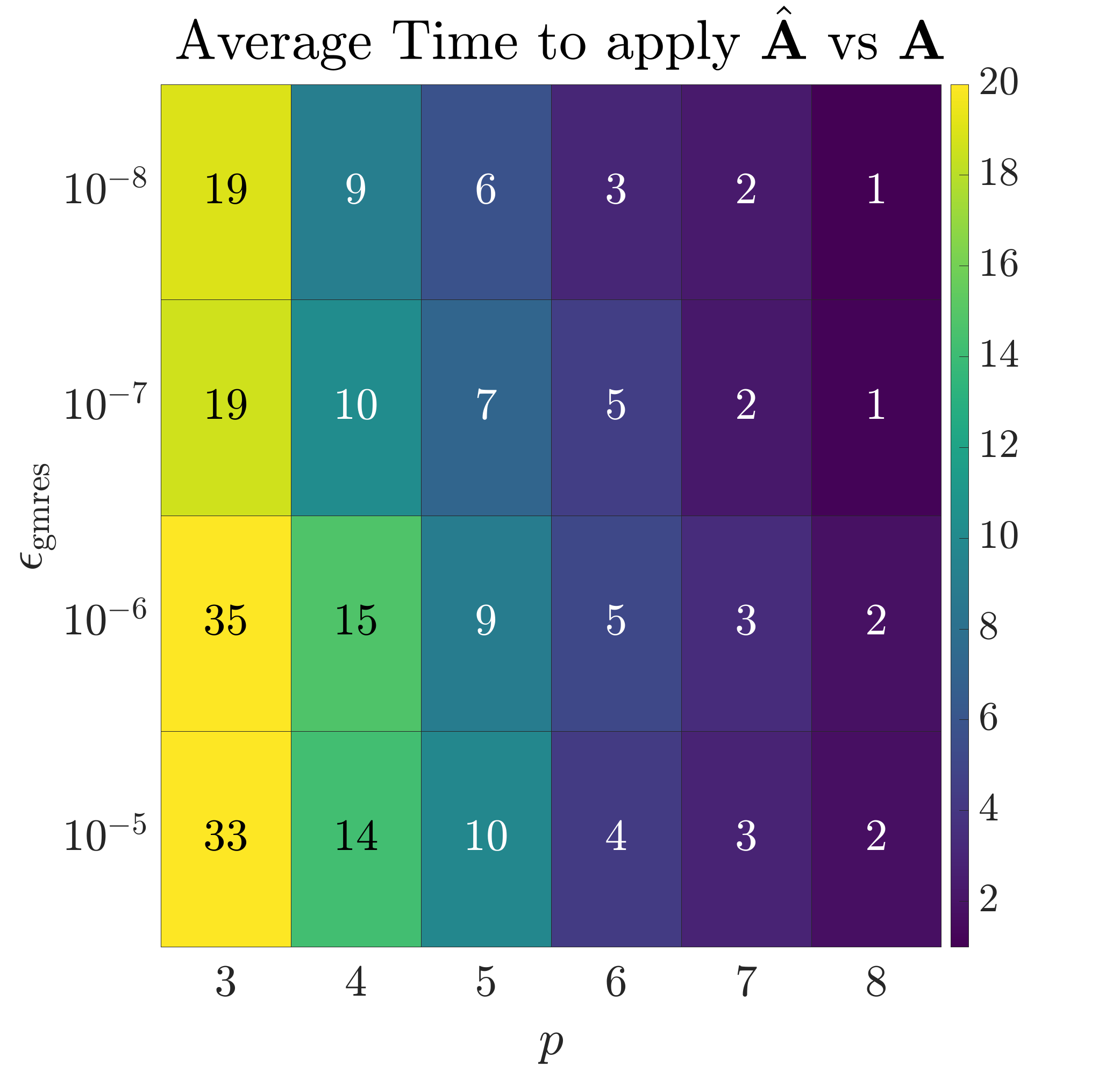}
    
    \caption{ 
    \textbf{Left:} Heatmap of $\log_{10}$ of the maximum (over 50 realizations) absolute error between low and high fidelity LCP matrices. \textbf{Right:} Heatmap displaying average number of low fidelity MVPs computed for the cost of one high fidelity MVP.} \label{fig:lofiMVPs}
\end{figure}
With these plots, values of \(p\) and \(\epsilon_\mathrm{gmres}\) can be identified that are promising candidates for low fidelity operators.
\begin{itemize}
    \item \textbf{Choosing \(p\)}: 
    Ideally, $p$ should be chosen to maximize a metric of accuracy per unit of computational cost; however, it is challenging to a priori and precisely translate this into what will minimize overall cost in our multifidelity algorithm. Instead, we opt for the following heuristic: we identify values of \(p\) such that digits of accuracy improve measurably compared to \(p - 1\), while their computational cost is relatively similar. These are \(p = 3, 4\) and \(6\). For each of these values of \(p\), the absolute error of the formed matrices is about \(3\) to \(4\) times less than for \(p - 1\), while the computational cost is $\lesssim 2$ times more expensive. In contrast, consider \(p = 5\), where the absolute error is only about \(3/4\) that of \(p = 4\). We did not consider \(p = 7\) as it did not offer significant computational savings compared to \(p = 8\).
    \item \textbf{Choosing \(\epsilon_\mathrm{gmres}\)}: Surprisingly, perhaps, computational cost did not appear to improve considerably as we increased the tolerance for the iterative linear solver $\epsilon_\mathrm{gmres}$. This result is likely due to the benign, low condition number characteristic of Fredholm BIE equations of the $2$nd kind and would have to be revisited in the presence of harsher conditioning. In our case, we therefore decided that a tight tolerance $\epsilon_\mathrm{gmres} \leq 10^{-6}$ was prudent and relatively efficient.
\end{itemize}
Thus, in this setting, our candidate low fidelity operators were \(p \in \{3,4,6\}\) and \(\epsilon_{\mathrm{gmres}} \leq 10^{-6}\). 

\subsection{A Bifidelity Proximal Quasi-Newton Method} \label{sec:multiFidelity-bpqn}

Assuming one selected a good candidate low fidelity MVP, we now present a novel method to efficiently solve \cref{eq:qp} incorporating both high and low fidelity MVPs. Importantly, in the presentation of this method, we do not make any assumptions upon the formulation of the low fidelity MVP \(\hat{\A}[\cdot]\), only that such a representation exists.

When choosing a quadratic model \(\B^{(k)}\) for a proximal quasi-Newton method, we have two important goals: the model should be as close to \(\A\) as possible and allow for efficient solutions to the proximal subproblem specified in \cref{eq:prox-pqn}. These two are often in conflict, and this can be seen in \cref{alg:pqn}; we have fast solves of the subproblem with \cref{prop-nnProx}, but we require \(\B^{(k)}\) to be a diagonal matrix plus a rank \(2r\) perturbation. We would like to find an alternative Hessian approximation that more closely resembles the true \(\A\) and can enable relatively fast applications of the weighted proximal operator. This is where the low fidelity MVP \(\hat{\A}\) can be utilized; choosing \(\B^{(0)} = \hat{\A}\) yields a substantial improvement in the in the accuracy of the local model compared to a diagonal matrix. Then, throughout optimization corrections to the local quadratic model can be made by satisfying high fidelity secant conditions
\begin{equation}\label{eq:bifi-qnapprox}
    \B^{(k)} = \hat{\A} + \U\U^\top - \V\V^\top
\end{equation}
as is done in a typical proximal quasi-Newton method.

This requires evaluating a weighted proximal operator at each iteration of the form in \cref{eq:prox-pqn}. Computing this expression directly employs computing \(\H^{(k)}[\nabla f(\x_k)]\); unfortunately, this is not feasible because the inverse of \(\hat{\A}\) is only available through the use of an iterative solver. Furthermore, \cref{prop-nnProx} is not applicable as \(\hat{\A}\) cannot be assumed to be diagonal. Thus, we minimize the quadratic model of \cref{eq:pqnQuadMdl} directly without using duality. Discarding terms that do not affect the location of the minimizer in \cref{eq:pqnQuadMdl}, one obtains a quadratic program
\begin{equation} \label{eq:pqnSubproblem}
    \hat{\x}^{(k)} = \argmin_{\z \geq \vec{0}} \tfrac{1}{2}\z^\top\B^{(k)}\z + (\A\xk +\b - \B^{(k)}\xk)^\top\z
\end{equation}
\begin{comment}
\end{comment}
 which can be solved with \emph{only low fidelity} MVPs. 
Naively, utilizing the rich local quadratic model induced from \cref{eq:bifi-qnapprox} can lead to a significant speedup when \cref{eq:pqnSubproblem} is solved with a first order method; we choose to use \Cref{alg:pqn}.  
However, this choice alone does not utilize all the structure in the subproblem. The gradient of the subproblem at iteration \(k\) will be of the form
\begin{equation*}
    \B^{(k)}\z + (\A\xk +\b - \B^{(k)}\xk)
\end{equation*}
and thus the secant information from the subproblem solve will be
\begin{equation*}
    \B^{(k)}\hat{\mat{S}}^{(k)}_{\text{low}} = \hat{\mat{Y}}^{(k)}_{\text{low}}.
\end{equation*}
This secant information can then be cached and reused in the next subproblem solve at iteration \(k + 1\), but the secant conditions should correspond to the next Hessian approximation \(\B^{(k+1)}\) and not the current \(\B^{(k)}\). To resolve this, recall that \(\B^{(k + 1)} = \B^{(k)} + \vec{u}^{(k)}(\vec{u}^{(k)})^\top -  \vec{v}^{(k)}(\vec{v}^{(k)})^\top\) and thus the updated secant information at \(k + 1\) can be found by computing 
\begin{align*}
    \mat{S}^{(k + 1)}_{\text{low}} &= \hat{\mat{S}}^{(k)}_{\text{low}} \\
    \mat{Y}^{(k + 1)}_{\text{low}} &= \hat{\mat{Y}}^{(k)}_{\text{low}} - \left(\vec{u}^{(k)}(\vec{u}^{(k)})^\top -  \vec{v}^{(k)}(\vec{v}^{(k)})^\top\right)\hat{\mat{S}}^{(k)}_{\text{low}}.
\end{align*}
The process repeats and thus secant conditions from every previous subproblem are accumulated for use in the current subproblem solve. This accelerates the convergence of the subproblem solver. Furthermore, we can use a \emph{low fidelity initialization} for the algorithm by solving the low fidelity version of \cref{eq:qp} 
\begin{equation*}
    \x^{(0)} = \argmin_{\x \geq \vec{0}} \frac{1}{2}\x^\top\hat{\A}\x + \b^\top\x
\end{equation*} approximately using \MPQN. \Cref{alg:bpqn} succinctly describes our \BPQN method; though, we omit describing the secant condition re-use in the subproblem. More technical details are available in \cref{app:bpqn}.

\begin{algorithm}[ht]
\caption{Bifidelity Proximal Quasi-Newton (\BPQN)} \label{alg:bpqn}
\begin{algorithmic}
\Procedure{BifidelityProximalQuasiNewton}{$\A[\cdot], \b, \x^{(-1)}, \hat{\A}; k_{\max}, \epsilon_{\mathrm{kkt}}$}
    \State $\x^{(0)}, \mat{S}_{\text{low}}^{(0)}, \mat{Y}_{\text{low}}^{(0)} \gets$ ProximalQuasiNewton$\bigl(\hat{\A}[\cdot], \b, \x^{(-1)})$
    \State $\B^{(0)} \gets \hat{\A}$
    \For{$k \gets 0:k_{\max}-1$}
        \If{\(\langle \xk[k], \A[\xk[k]] + \b\rangle < \epsilon_{\mathrm{kkt}}\)}
            \State \Return $\xk[k]$
        \EndIf
        \State $\hat{\x}^{(k)}, \hat{\mat{S}}_{\text{low}}^{(k)}, \hat{\mat{Y}}_{\text{low}}^{(k)} \gets$ ProximalQuasiNewton$\bigl(\B^{(k)}[\cdot], (\A[\xk] + \b - \B^{(k)}\xk), \xk, \mat{S}_{\text{low}}^{(k)}, \mat{Y}_{\text{low}}^{(k)}\bigr)$ %\Comment{See \cref{alg:pqn}}
        \State $\p^{(k)} = \hat{\x}^{(k)} - \xk$
        \State $\eta^{(k)} \gets$ StepSizeSelection($\A[\cdot], \b, \xk, \p^{(k)}$) \Comment{See \cref{alg:stepSize}}
        \State $\xk[k+1] \gets \xk + \p^{(k)} \eta^{(k)}$ 
        \State $\B^{(k)}, \H^{(k)} \gets$ quasiNewtonUpdate($\xk[k+1] - \xk$, $\eta^{(k)} \A[\p^{(k)}]$) \Comment{See, e.g., \cite[Ch.~6]{NocedalWright2006}}
        \State $\mat{S}_{\text{low}}^{(k + 1)} \gets \hat{\mat{S}}^{(k)}_{\text{low}}$ 
        \State $\mat{Y}_{\text{low}}^{(k + 1)} \gets \hat{\mat{Y}}^{(k)}_{\text{low}} - \left(\vec{u}^{(k)}(\vec{u}^{(k)})^\top -  \vec{v}^{(k)}(\vec{v}^{(k)})^\top\right)\hat{\mat{S}}^{(k)}_{\text{low}}$
    \EndFor
    \State \Return $\xk[k+1]$
\EndProcedure
\end{algorithmic}
\end{algorithm}
As \cref{alg:bpqn} is a proximal Newton-type method that uses the optimal step size, the same convergence guarantees apply to \BPQN as they apply to \MPQN. Thus, under mild assumptions, \BPQN is globally convergent via \cref{lemma:convergence}. The discussion regarding linear and superlinear convergence of \BPQN is the same as it was for \MPQN; theoretical linear and superlinear convergence are obtainable but the modifications needed lead to a practically worse algorithm.

\section{Numerical Results} \label{sec:res}

In this section, a series of experiments demonstrate the performance of our methods in addressing the computational bottleneck posed by optimization-based collision resolution in dense Stokesian suspensions, as compared to state-of-the-art methods. In line with the guiding principle stated in our introductory section, the metric of interest to compare algorithmic performance is set to the number of \emph{effective matrix vector products} (E-MVPs). For algorithms that only utilize the high fidelity MVP $\A[\cdot]$, then E-MVPs is the same as MVPs. For \BPQN, the cost of low fidelity MVPs is converted to E-MVPs by scaling the number of low fidelity MVPs by the ratio between low and high fidelity MVP costs, and added to the number of high fidelity MVPs. This ratio is re-calculated for every timestep to more accurately account for its variation with respect to particle configurations.

\paragraph{Experimental Set Up} 
All experiments in this work involve the simulation of dense suspensions of Janus particles, following the formulations in \cite{KohlCoronaCheruvuEtAl2021arXiv210414068}. Janus particles, as their name suggests, are characterized by anisotropic structure involving two hemispheres that differ in electric, magnetic, hydrophilic, etc.\ properties.  Notably, dense suspensions of Janus particles demonstrate complex aggregate behavior such as clustering \citep{HongCacciutoLuijtenEtAl2008Langmuir}, self-assembly \citep{JiangGranick2012}, and formation of larger-scale structures \citep{WaltherMuller2008SoftMatter}. Large-scale simulation of Janus particle systems is arguably an ideal testbed, as their characteristic behavior induces persistently large collisional problems, and each MVP involves an expensive solve of a Stokes mobility problem. 

We demonstrate the practical performance of our algorithms on two sets of experiments. First, we ran a simulation of $\Np=125$ Janus particles in a lattice ($5\times5\times5$) for 250 time steps. The particles were spaced such that they were close, but very few collisions would occur at the initial time. The amphiphilic tails of the particles are initialized to point towards the origin which causes the particles to slowly move towards each other; at time steps 201-250 the particles' configuration is tightly packed, and the collision resolution problems' sizes are all such that $\Nc \gtrapprox \Np$, as particles cluster to hide hydrophobic regions from the fluid. Then, using the saved configurations from each time step, we generated dense representations of $\A$ and $\b$ in a post process for time steps 201-250. This allows us to test \emph{offline} which facilitated comparison of a wide selection of methods. These results are shown in \cref{sec:res-offline}. 

Next, in \cref{sec:res-online} our algorithms are compared in an \emph{online} setting. In this test case, we run experiments on several different sized lattices for \(200\) time steps to test whether our method scales favorably as  $\Np$ (and thus $\Nc$) increases. The initial configuration is created so that $\mathcal{O}(\Np)$ collisions will be considered in the first time step. Full details of how the initial configuration is generated are specified in \cref{app:initConf}. 
 
All experiments in this work employ a serial Matlab implementation of the BIM-based mobility solver in \cite{CoronaVeerapaneni2018JournalofComputationalPhysics,KohlCoronaCheruvuEtAl2021arXiv210414068} with an openMP Stokes FMM from \citep{GimbutasGreengard2023STKFMMLIB3D} for far field interactions. The simulation used for the densely saved LCP matrices and the offline comparison was run on the Blanca condo cluster \citep{universityofcoloradoboulderresearchcomputingBlancaCondoCluster2021} using AMD Milan CPUs. Once saved, the offline comparison was run on a M3 Apple Macbook Pro. Online simulations were run on the Alpine cluster \citep{universityofcoloradoboulderresearchcomputingAlpine2023} using 24 cores of an AMD Milan CPU with 400 GB of RAM.

\paragraph{Hyperparameters and Convergence}
All solvers use \((p=8, \epsilon_{\mathrm{gmres}} = 10^{-8})\) for the high fidelity $\A$, and \BPQN utilize \((p \in \{3,4,6\}, \epsilon_{\mathrm{gmres}} \leq 10^{-6})\) for the low fidelity \(\hat{\A}\). LCP solvers are considered to have converged when the absolute or relative KKT error is less than a tolerance of $10^{-8}$, where these errors are defined as  
\begin{equation}\label{eq:kkt-error}
    \epsilon_{\mathrm{kkt}_{\mathrm{abs}}}^{(k)} = \left\|\min(\x, \A[\x]+\b)\right\|, \; \epsilon_{\mathrm{kkt}_{\mathrm{rel}}}^{(k)} = \frac{\bigl\lvert\epsilon_{\mathrm{kkt}_{\mathrm{abs}}}^{(k)} - \epsilon_{\mathrm{kkt}_{\mathrm{abs}}}^{(k-1)}\bigr\rvert}{\max\left(\epsilon_{\mathrm{kkt}_{\mathrm{abs}}}^{(k)},\epsilon_{\mathrm{kkt}_{\mathrm{abs}}}^{(k-1)}\right)}.
\end{equation}
We found consistent results for slightly smaller and larger tolerances, and we selected a tolerance $10^{-8}$ because it leads to physical behavior in our simulations. 

\paragraph{Comparison Methods} \label{sec:res-compMeth}

Gradient descent and its accelerated variants are the simplest and most ubiquitous gradient based methods for unconstrained optimization; proximal gradient descent\footnote{Elsewhere this method is referred to as projected gradient descent, but it is often preferable to understand projected gradient descent as a special case of proximal gradient descent.} (PGD) is the corresponding method for constrained optimization. When specialized to \cref{eq:qp}, convergence is guaranteed for a positive constant step size $\tau^{(k)}=\tau < 2/\|\A\|_2$ \citep{NocedalWright2006}; this can, however, be overly conservative. It is also possible to use line searches, e.g., backtracking. Because MVP's are at a premium, estimating the spectral norm or performing a line search are both cost prohibitive, as they 
require many additional evaluations of $\A[\cdot]$. For this reason, a heuristic step size called the Barzilai-Borwein or spectral step size is appealing:
\begin{equation} \label{eq:spectralStepSize}
\tau_{\mathrm{bb}} = \frac{\bigl\|\xk - \x^{(k-1)}\bigr\|^2}{\bigl(\xk - \x^{(k-1)}\bigr)^\top\bigl(\A [\xk] - \A [\x^{(k-1)}]\bigr)}.
\end{equation}
This step size can be interpreted as a $0^\mathrm{th}$ order approximation of the Hessian and has seen practical success. We refer to this as BB-PGD; we found this step size was the best performing method to solve \cref{eq:qp} in \cite{YanCoronaMalhotraEtAl2020JournalofComputationalPhysics}. Unfortunately, it has few theoretical guarantees without making very strong assumptions and, in fact, counter examples of quadratic objectives with box constraints can be constructed such that BB-PGD will not converge~\citep{dai2005projected}. This method serves as a baseline to compare our methods to because it is currently the state-of-the-art for BIM formulations in dense regimes.

The accelerated variant of proximal gradient descent (A-PGD) is another comparison-worthy algorithm. Optimization theory \citep{Nesterov1983ProcUSSRAcadSci, Nesterov2013MathProgram, BeckerBobinCandes2011SIAMJImagingSci, BeckTeboulle2009SIAMJImagingSci} would suggest that A-PGD would converge faster than PGD since it has a better dependence on the condition number, but this requires a line search or knowledge of the Lipschitz constant of the gradient which, again, is not available without additional MVPs. We experimented with heuristics such as restarting or a coarse estimate of the Lipschitz constant but found no strategy that was competitive. The implementation and discussion in \citep{goldsteinFieldGuideForwardBackward2016} were used as a starting point for our implementation. In \cref{sec:res-offline}, A-PGD is given a \textit{known} Lipschitz constant $L(\A)$ and strong convexity constant $\mu(\A)$ to demonstrate that even in an unrealistically generous circumstance, A-PGD is not competitive. Our findings corroborate \citep{YanCoronaMalhotraEtAl2020JournalofComputationalPhysics} who observed that A-PGD was outperformed by BB-PGD. One possible explanation is that A-PGD improves the dependence on the condition number $\kappa$ from $\mathcal{O}(1-1/\kappa)$ to $\mathcal{O}(1-1/\sqrt{\kappa})$, but our LCP problems are well-conditioned, so for $\kappa = \mathcal{O}(10)$ the square-root improvement is minor. Also, because A-PGD often converged in around 15 iterations, it may be that the benefits of momentum can not be fully realized in such few iterations.

We also show results for two other proximal quasi-Newton methods: zeroSR1 and the well-known L-BFGS-B. The zeroSR1 algorithm is a zero memory variant that facilitates the weighted proximal operator to be evaluated in $\mathcal{O}(\Nc\log(\Nc))$ time \citep{BeckerFadiliOchs2018arXiv180108691}; while this is an advantage, zeroSR1 does not benefit from significant curvature information because it has no memory. L-BFGS-B \citep{ByrdLuNocedalEtAl1995SIAMJSciComput} uses a similar algorithmic design to our \MPQN algorithm \citep{LeeSunSaunders2014SIAMJOptim}, but L-BFGS-B does not evaluate the weighted proximal operator in the same way. It solves this subproblem with an active set method, while we solve a special dual problem with a semi-smooth algorithm. Also, our algorithm uses a fixed step size of 1 and an optimal over-relaxation parameter $\eta^*$. In contrast, L-BFGS-B performs a line search over only the step size.

The only second order method we compare to is Min(imum)-Map Newton \citep{Anitescu2006MathProgram}. Second order methods require solving a linear system at each iteration; this necessitates using many expensive MVPs with a Krylov solver. Min-Map Newton is included for completeness to demonstrate our focus on first order methods. \Cref{tab:comparisonMethods} summarizes the methods used for comparison.

\begin{table}[ht!]
    \centering
    \begin{tabular}{llp{3.5cm}p{1.6cm}p{1.6cm}}
        \toprule
        & Algorithm & Derivative information  & Step size & Over-relaxation Parameter \\
        \midrule
        \multirow{5}{*}{Baseline Methods} &  BB-PGD & First Order  & $\tau_\mathrm{bb}$ & 1\\
        & A-PGD & First Order & $1/L(\A)$ & 1\\
        & zeroSR1 & First Order & 1 & 1\\
        & L-BFGS-B & First Order & line search & 1\\
        & Min-Map Newton & Second Order & 1 & NA\\
        \midrule
        \multirow{2}{*}{Proposed Methods} & \MPQN & First Order & 1 & $\eta^*$\\
       & \BPQN & First Order & 1 & $\eta^*$\\
        \bottomrule
    \end{tabular}
    \caption{For each algorithm, we specify the derivative information necessary, the step size, and over-relaxation parameter where applicable. Notably, our algorithms do not utilize second order information or a line search.}
    \label{tab:comparisonMethods}
\end{table}

\subsection{Offline Comparison}\label{sec:res-offline}

The dense representations of $\A$ facilitate a broader comparison of algorithms due to some methods being computationally infeasible in ``the loop". All methods compute one MVP before iterations begin to check if the convergence is met; thus, the number of MVPs is greater than the number of iterations plus one. When a second order method or a line search strategy is used, the number of MVPs can be even more. In particular, Min-Map Newton often converges in as few as two iterations, but it requires a considerable number of MVPs because of the Newton system solve. \Cref{tab:offline} summarize the results from these simulations.
\begin{table}[ht!]
    \centering
    \resizebox{\textwidth}{!}{
        \begin{tabular}{llccccc}
\toprule
 & & Minimum & Median & Mean & Maximum \\
\midrule
% $\text{\emph{Baseline Methods}}$ &  &  &  &  &  &  \\
\multirow{5}{*}{Baseline Methods} & BB-PGD & 11 & 11 & 11 & 12 \\
& A-PGD & 12 & 15 & 15 & 16 \\
& zeroSR1 & 9 & 13 & 13 & 15 \\
& L-BFGS-B & 10 & 11 & 11 & 12 \\
& Min-Map Newton & 15 & 28 & 27 & 47 \\
\midrule
\multirow{4}{*}{Proposed Methods} & Mono-PQN & 7 & 8 & 7.8 & 8 \\
& B-PQN : $p=3, \epsilon_\mathrm{gmres}=10^{-6}$ & 4.7 (4) & 5.4 (4) & 5.5 (4.4) & 7.1 (5) \\
& \textbf{B-PQN : $\mathbf{p=4, \epsilon_\mathrm{gmres}=10^{-6}}$} & \textbf{3.8 (3)} & \textbf{4.6 (3)} & \textbf{4.6 (3)} & \textbf{5.5 (4)} \\
& B-PQN : $p=6, \epsilon_\mathrm{gmres}=10^{-6}$ & 3.8 (2) & 5.1 (2) & 5.4 (2.1) & 8.4 (3) \\
\bottomrule
\end{tabular}

    }
    \caption{Each method is given the same set of 50 LCPs. The number of effective MVPs necessary to satisfy a stationarity condition is recorded. We include the number of high-fidelity MVPs for the \BPQN in parentheses to demonstrate that convergence is satisfied in as few as one to four iterations.}
    \label{tab:offline}
\end{table}

The well known first order methods BB-PGD and A-PGD both out perform Min-Map Newton as expected.  All proximal quasi-Newton methods are competitive, but our \MPQN method outperforms all other algorithms, leading to an $\approx 1.5\times$ speed up compared to BB-PGD. Notably, L-BFGS-B requires more MVPs than our \MPQN; even though L-BFGS-B optimizes a similar local quadratic model, our
\MPQN outperforms because of design choices made to specialize for \cref{eq:qp}. Our \MPQN outperforms zeroSR1 primarily due to having a richer curvature information. \BPQN substantially outperforms all methods. Its success stems from leveraging a highly informative local quadratic model with an algorithmic design again tailored to \cref{eq:qp}. 

\subsection{Online Comparison} \label{sec:res-online}
The baseline algorithm of BB-PGD is compared to our \MPQN and \BPQN algorithms. Each algorithm is run on increasingly larger square lattices: $(3\times3\times3, \Np = 27), (4\times4\times4,\Np = 64),(5\times5\times5, \Np = 125)$ and $(6\times6\times6, \Np = 216)$. In \cref{fig:online}, the number of E-MVPs for each algorithm is plotted against the size of the lattice. Both of our algorithms' performances dominate the baseline of BB-PGD. During online testing, \BPQN utilizes \((p = 4, \epsilon_{\mathrm{gmres}} = 10^{-8})\) for the low fidelity \(\hat{\A}\).

\begin{figure}[ht!]
    \centering 
    \includegraphics[width=\textwidth]{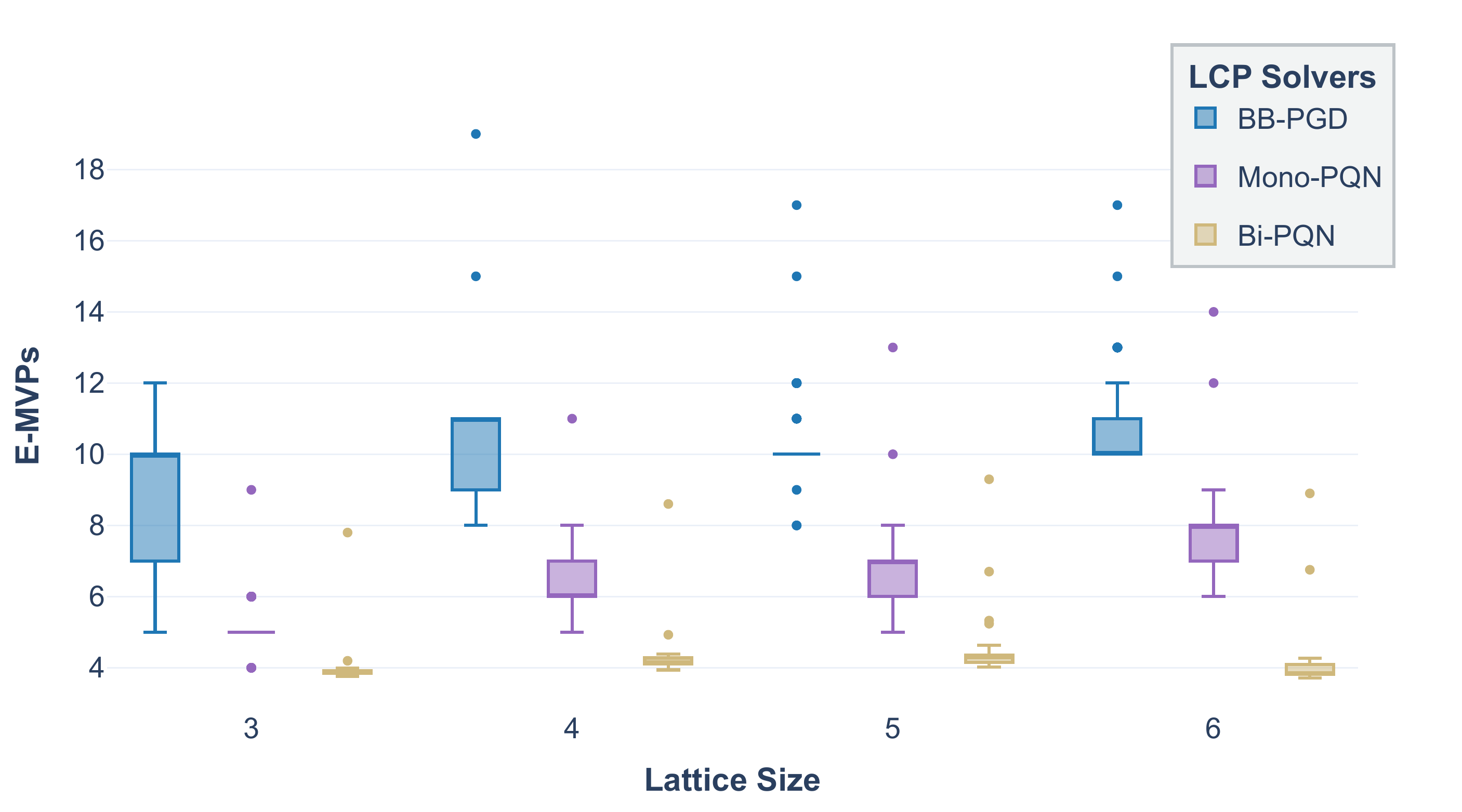}
    \caption{For each online experiment defined by lattice size on the x-axis, we compare box plots of each candidate algorithm's E-MVP statistics across $200$ simulation timesteps. \MPQN and \BPQN consistently outperform PGD, with \BPQN showing nearly problem size independent scaling.} 
    \label{fig:online}
\end{figure}

\begin{table}[ht!] 
    \centering
    \begin{tabular}{lcccc}%
\toprule%
& \multicolumn{4}{c}{Lattice size} \\
\cmidrule(l){2-5}
&3&4&5&6\\%
\midrule%
\textbf{BB{-}PGD}&&&&\\%
\hspace{.25cm} Total Run Time&15h 21m&1d 3h 54m&2d 18h 36m&7d 18h 0m\\%
\hspace{.25cm} Collision Resolution&13h 31m&23h 26m&2d 0h 43m&5d 0h 25m\\%
\textbf{Mono{-}PQN}&&&&\\%
\hspace{.25cm} Total Run Time&$1.59\times$&$1.37\times$&$1.32\times$&$1.16\times$\\%
\hspace{.25cm} Collision Resolution&$1.70\times$&$1.49\times$&$1.43\times$&$1.32\times$\\%
\textbf{Bi{-}PQN}&&&&\\%
\hspace{.25cm} Total Run Time&$1.85\times$&$1.75\times$&$1.60\times$&$1.57\times$\\%
\hspace{.25cm} Collision Resolution&$2.07\times$&$2.10\times$&$2.10\times$&$2.47\times$\\%
\bottomrule%
\end{tabular}
    \caption{The total run time and the time spent during collision resolution is recorded for the baseline algorithm (top rows) and the respective speed-ups are recorded for our proposed algorithms. Notably our algorithms show a speed up of $\approx 1.5 \times $ for \MPQN and $> 2\times$ for \BPQN.} \label{tab:online}
\end{table}

In \cref{tab:online}, one can see that \MPQN outperforms BB-PGD for all problem sizes, but relative speed up of the collision resolution stage decreases as problem sizes get larger. In contrast, \BPQN demonstrates a more dramatic $> 2\times$ speed up. Most importantly, \BPQN's computational advantage is maintained as the problem size increases; notice that in \cref{fig:online} the distribution of E-MVPs for \BPQN appears to be practically constant as a function of problem size, and has considerably less variance as compared to competing methods. This makes \BPQN's performance particularly promising, as it suggests that \BPQN will scale better for larger, more challenging problems. We hypothesize the observed performance suggests \BPQN steps are similar enough to true proximal Newton steps that we see the small iteration counts typical of second order methods without their associated computational cost.

\section{Conclusions and Future Work}\label{sec:disc}

This work was motivated by the identification of collision resolution as a computational bottleneck in large-scale dense suspension simulations; in particular, the number of MVPs used during optimization. Prior work by \citep{YanCoronaMalhotraEtAl2020JournalofComputationalPhysics} identified BB-PGD, a line search free method, to be the de facto method for large scale collision resolution. Motivated by this finding, we designed two custom proximal quasi-Newton methods. Notably, our \MPQN with an optimal over-relaxation parameter and an efficient weighted proximal operator solver shows a \(\approx 1.5 \times\) speed up of collision resolution over BB-PGD. In addition, our \BPQN leverages a low fidelity Hessian approximation and cleverly tailored implementation of \MPQN to obtain a \(>2\times\) speed up over the baseline. Both algorithms provide practically significant accelerations to online simulation of particulate systems of interest. In particular, \BPQN consistently converges in three high fidelity gradient evaluations across problem sizes and setups considered in this work, which is arguably near optimal and so, further acceleration seems unlikely. 

In PDE-based methods such as the BIM approach used in this work, every MVP of \(\A\) (and \(\hat{\A}\)) involves a GMRES linear solve. As iterative optimization methods utilize repeated MVPs, this setting is particularly suited to an application of GMRES subspace recycling. We aim to incorporate these in ongoing work, and test their contribution to the methods performance. 

There are a number of potential algorithmic developments and generalizations of interest for the multifidelity framework in our \BPQN method. Beyond its adaptation to various formulations of the mobility LCP and associated options for multifidelity, a general avenue of algorithmic extension could involve recursive or hierarchical use of fidelity to generalize the \BPQN: in our approach, the low fidelity subproblem is itself a non-negativity constrained quadratic program in which the quadratic term is the result of a mobility solve. Thus, the low fidelity problem itself is amenable to \BPQN. 

Finally, extending the work presented to a broader range of large-scale simulation problems is a natural avenue of future research. While this work only considered spherical particles with a no-slip boundary conditions, our framework could be adapted to handle non-spherical (e.g. spheroidal, convex, poly-convex) particle geometries, confinement and different boundary conditions. As the proposed optimization methods are developed to address computational bottlenecks due to collision resolution, cases of practical interest leading to challenging optimization problems (i.e., poor conditioning of \(\A\)) would be of particular interest. Our methods could be extended to frictional contact problems involving more general, non-linear complementarity problems (e.g. conic constraints in the dry granular case \cite{TASORA2011439,DeCoronaJayakumarEtAl20192019IEEEHighPerformExtremeComputConfHPEC}). As discussed in \cite{Brown2013Physics}, while tangential contact field modeling is a subject of substantial debate in the literature, there is evidence to suggest frictional particle collisions play an important role in abrupt material state transitions such as discontinuous shear thickening.

\section{Acknowledgements}
This material is based upon work supported by the National Science Foundation under Award No. 2309661 and the Department of Energy, Office of Science, Advanced Scientific
Computing Research under Award Number DE-SC0023346.

This work utilized the Alpine high performance computing resource at the University of Colorado Boulder. Alpine is jointly funded by the University of Colorado Boulder, the University of Colorado Anschutz, Colorado State University, and the National Science Foundation (award 2201538) \citep{universityofcoloradoboulderresearchcomputingAlpine2023}.
This work utilized the Blanca condo computing resource at the University of Colorado Boulder. Blanca is jointly funded by computing users and the University of Colorado Boulder \citep{universityofcoloradoboulderresearchcomputingBlancaCondoCluster2021}.

\bibliographystyle{amsalpha}
\bibliography{references}

\appendix 
\section{Fast Solutions to Weighted Proximal Operator} \label{app:solveProx}

The work \citep{BeckerFadiliOchs2018arXiv180108691} gives a complete treatment of this subject, but for completeness and clarity, we specialize their results to the case of a strongly convex quadratic program with non-negativity constraints.

\subsection{Discussion of the Special Dual Problem}\label{app:solveProx-deriv}

For ease of presentation, the superscripts referring to the iteration $k$ are dropped, and the weighted proximal operator is evaluated at an arbitrary vector $\z\in \mathbb{R}^n$. The current approximation of the Hessian is 
\[\B = \D +\U\U^\top - \V\V^\top\]
where $\D\succ 0$ is diagonal and $\U, \V \in \mathbb{R}^{n\times r}$ for some $r\leq m$.

For full proofs of the results necessary to obtain Proposition~\ref{prop-nnProx} we refer the reader to Appendix A of \citep{BeckerFadiliOchs2018arXiv180108691}. Here, we highlight that two different notions of duality, common convex analysis identities and the Fr\'echet differentiability of the Moreau envelope are utilized in order to obtain their result; see the textbook \citep{BauschkeCombettes2017} for background on these concepts. First, Fenchel duality is used to solve the weighted proximal operator $\prox_{\iota_C}^{\D +\U\U^\top}$ under the metric induced by the sum of a diagonal matrix $\D$ and the low rank update $\U\U^\top$. Then, Toland duality leverages the full convexity of the difference of two convex functions $\tilde{h}(\z) = \z^\top (\D +\U\U^\top)\z$ and $\tilde{g}(\z) = \z^\top (\V\V^\top)\z$. The formal statements of these dual problems are stated below in Lemmas \ref{lemma:fenchel} and \ref{lemma:toland} for the readers convenience\footnote{The notation $g^*$ denotes the Fenchel-Legendre conjugate of $g$, and $\partial g$ denotes the subdifferential.}. The functions are allowed to take on values in $\mathbb{R}\cup\{+\infty\}$ in order to encode the domain, so that $\dom f = \{ \z \mid f(\z) < \infty \}$.
\begin{lemma}[Fenchel duality (Proposition 15.13 of \citep{BauschkeCombettes2017})] \label{lemma:fenchel} 
    Let $h$ and $g$ be convex and lower semicontinuous. Suppose that $0 \in \operatorname{ri}(\operatorname{dom} g-\operatorname{dom} h)$\footnote{$\operatorname{ri}(\cdot)$ defines the relative interior of a set.}. Then
    \[\inf _{\z } h(\z)+g(\z)=-\min _{\w} h^*(-\w)+g^*(\w),\]
    with the extremality relationships between $\z^{\star}$ and $\w^{\star}$, respectively the solutions of the primal and dual problems
    \[\z^{\star} \in \partial h^*\left(-\w^{\star}\right) \text{ and } \quad \w^{\star} \in \partial g\left(\z^{\star}\right).\]
\end{lemma}

\begin{lemma} [Toland duality] \label{lemma:toland}
    Let $\tilde{h}$ and $\tilde{g}$ be convex and lower semicontinuous. Then
    \[\inf _{\tilde{\z}\in \dom \tilde{h}} \tilde{h}(\tilde{\z})-\tilde{g}(\tilde{\z})=\inf _{\tilde{\w}\in\dom \tilde{g}^* } \tilde{g}^*(\tilde{\w})-\tilde{h}^*(\tilde{\w}) .\]
    If $\tilde{h}-\tilde{g}$ is coercive, and $\tilde{\w}^{\star}$ solves the dual problem in $\tilde{\w}$, then there exists a solution $\tilde{\z}^{\star}$ of the primal problem and
    \[\begin{array}{ll}
    \tilde{\z}^{\star} \in \partial \tilde{h}^*\left(\tilde{\w}^{\star}\right) & \text { and } \quad \tilde{\w}^{\star} \in \partial \tilde{g}\left(\tilde{\z}^{\star}\right), \\
    \tilde{\w}^{\star} \in \partial \tilde{h}\left(\tilde{\z}^{\star}\right) & \text { and } \quad \tilde{\z}^{\star} \in \partial \tilde{g}^*\left(\tilde{\w}^{\star}\right) .
    \end{array}\]
\end{lemma}
The first statement follows from either \citep[Corollary 14.20]{BauschkeCombettes2017} or Theorem 2.2 of Toland's original paper \citep{Toland1979ArchRationalMechAnal}, and the second set of statements follow by combining Theorems 2.7 and 2.8 in \citep{Toland1979ArchRationalMechAnal}.

Notably, Toland duality is in general a vacuous statement because the difference of the functions $\tilde{h}$ and $\tilde{g}$ may not have a global minimum. In the particular case, where $\tilde{h}-\tilde{g}$ is known to be convex \emph{a priori}, then this is a powerful tool. For BFGS updates, $\tilde{h}-\tilde{g}$ corresponds to the norm induced by hessian approximation $\B$ which is guaranteed to be a convex function by construction. This gives Toland duality the conditions it requires to provide a minimizer. 

The key benefit of the dual formulations is that the dual problem minimization occurs over a subspace that can be parameterized by the columns of $\U$ and $\V$. This is what decreases the size of the dual problem leading to the $2r$ dimensional root finding problem as opposed to the $\Nc$ dimensional primal problem.

\subsection{Numerical Approximation via Semi-Smooth Newton} \label{app:solveProx-ssnewton}
Recall the piecewise linear function defined in \cref{eq:prop1-rootFinding}
\begin{align*}
    \mathcal{L}\left(\alphavec,\tilde{\alphavec}\right) 
    =\begin{bmatrix}
        \U^\top \Bigl[\tilde{\x} + \C^{-1}\V\tilde{\alphavec} - \bigl(\tilde{\x} + \C^{-1}\V\tilde{\alphavec} - \D^{-1}\U\alphavec\bigr)_+\Bigr] \\
        \V^\top \Bigl[\tilde{\x} - \bigl(\tilde{\x} + \C^{-1}\V\tilde{\alphavec} - \D^{-1}\U\alphavec\bigr)_+\Bigr]
    \end{bmatrix}
\end{align*}
where $\C = \D + \U\U^\top$. We can define the convenience function
\[\boldsymbol{\Lambda}_{ii}\left(\alphavec,\tilde{\alphavec}\right) = \begin{cases}
    1,& \Bigl(\bigl(\tilde{\x} + (\D + \U\U^\top)^{-1}\V\tilde{\alphavec} - \D^{-1}\U\alphavec\bigr)_+\Bigr)_{i} > 0\\
    0,& \text{otherwise}
\end{cases}\]
which allows us to efficiently define the Jacobian of $\mathcal{L}(\alphavec)$ almost everywhere
\begin{align*}
    J_\mathcal{L} (\alphavec, \tilde{\alphavec}) &= \begin{bmatrix} 
        \U^{\top}\boldsymbol{\Lambda}\left(\alphavec,\tilde{\alphavec}\right)  \D^{-1} \U & \U^{\top} \left[\mat{I}_n - \boldsymbol{\Lambda}\left( \alphavec, \tilde{\alphavec}\right)\right] \C^{-1} \V\\
        \V^{\top} \boldsymbol{\Lambda}\left(\alphavec,\tilde{\alphavec}\right)  \D^{-1} \U & -\V^{\top}\boldsymbol{\Lambda}\left(\alphavec,\tilde{\alphavec}\right) \C^{-1} \V
    \end{bmatrix} + \mat{I}.
\end{align*} 
This root finding problems falls under the class of semi-smooth functions \citep{Mifflin1977SIAMJControlOptim}, and \cref{alg:semiSmoothNewton} can approximate the solution with local super-linear convergence \citep{QiSun1993MathematicalProgramming, Hintermuller} which we observe in practice.
\begin{algorithm}
\caption{SemiSmooth Newton} \label{alg:semiSmoothNewton}
\begin{algorithmic}
\Procedure{SemiSmoothNewton}{$\D, \U, \V, \tilde{\x}; j_{\max}, \epsilon_{\mathrm{abs}}$}
\Comment{$\D\in\mathbb{R}^{\Nc\times \Nc}$, $\U,\V\in\mathbb{R}^{\Nc\times r}$, $\tilde{\x}\in\mathbb{R}^\Nc$}
    \State $[\alphavec_0\; \tilde{\alphavec}_0]^\top \gets \vec{0}$
    \Comment{$\alphavec_j, \tilde{\alphavec}_j\in\mathbb{R}^{1\times r}$}
    \For{$j \gets 0:j_{\max}-1$}
        \State $[\alphavec_{j+1}\; \tilde{\alphavec}_{j+1}]^\top \gets [\alphavec_j\; \tilde{\alphavec}_j]^\top - J_{\mathcal{L}}^{-1}\bigl([\alphavec_j\; \tilde{\alphavec}_j]^\top\bigr) \mathcal{L}\bigl([\alphavec_j \; \tilde{\alphavec}_j]^\top\bigr)$
        \If{\(\left\|[\alphavec_{j+1}\; \tilde{\alphavec}_{j+1}] - [\alphavec_{j}\; \tilde{\alphavec}_{j}]\right\| < \epsilon_{\mathrm{abs}}\)}
            \State break
        \EndIf
    \EndFor
    \State \Return $\hat{\x} = \left(\tilde{\x} + \begin{bmatrix} 
        - \D^{-1} \U & \C^{-1} \V 
    \end{bmatrix} 
    \left[ \alphavec_j \tilde{\alphavec}_{j+1}\right]^\top\right)_+ $
\EndProcedure
\end{algorithmic}
\end{algorithm}

\section{Step Size Selection} \label{app:stepSize}
To prove \cref{lemma:convergence}, we need \cref{th:convergenceLee} from \citep{LeeSunSaunders2014SIAMJOptim}. We state a version that includes assumptions found elsewhere in the paper and somewhat specialized to this work.
\begin{theorem}[Theorem 3.1 \cite{LeeSunSaunders2014SIAMJOptim}]
\label{th:convergenceLee}
Suppose $f$ is a \(L\)-strongly smooth convex function, \(g\) is a lower semi-continuous convex function, and $\inf_{\x}\{f(\x) + g(\x)|\x \in \mathrm{dom}\ f + g\}$ is attained. If, in a proximal Newton-type method, $\B^{(k)} \succeq \mu \mat{I}$ for some $\mu > 0$, \(\tau^{(k)} = 1\), \(\eta^{(k)} > \eta_{\text{min}} > 0\), the iterates satisfy a \emph{sufficient decrease condition} 
\begin{equation}
    f(\x^{(k + 1)}) + g(\x^{(k + 1)}) \leq f(\x^{(k)}) + g(\x^{(k)}) + \alpha \eta^{(k)} \lambda
\end{equation}
where \(\lambda = \nabla f(\x^{(k)})^\top \p^{(k)} + g(\x^{(k)} + \p^{(k)}) - g(\x^{(k)})\), \(\alpha \in (0, \frac{1}{2})\), and the subproblems are solved exactly, then $\x^{(k)}$ from said proximal Newton-type method converges to an optimal solution starting at any $\x^{(0)} \in \mathrm{dom}\ f + g$. 
\end{theorem}

\stepSizeLemma*
\begin{proof}
Let \(f\) and \(g\) be as defined previously in \cref{sec:pqn}. For our objective \(f + g\), all the necessary regularity assumptions of \cref{th:convergenceLee} are met. The uniform strong convexity of the approximate Hessians and the minimum step size are satisfied by the assumption in \cref{lemma:convergence}. Thus, we only need to show that the sufficient decrease condition holds for the optimal step size computed by \cref{alg:stepSize}. As the optimal step size and \MPQN will only return feasible iterates, we can ignore the non-smooth terms of the sufficient decrease condition. Therefore, the sufficient decrease condition can be written as 
\begin{equation*}
f(\x^{(k + 1)}) \leq f(\x^{(k)}) + \alpha\eta\left((\nabla f(\x^{(k)})^\top\p^{(k)} \right)
\end{equation*}
which, for our $f$, becomes 
\begin{equation*}
\frac{1}{2}(\x^{(k)} + \eta \p^{(k)})^\top \A (\x^{(k)} + \eta \p^{(k)}) + \b^\top(\x^{(k)} + \eta\p^{(k)}) \leq \frac{1}{2}(\x^{(k)})^\top\A\x^{(k)} + \b^\top\x^{(k)} + \alpha\eta\left(\A\x^{(k)} + \b\right)^\top\p^{(k)}.
\end{equation*}
Simplifying yields an equivalent condition of
\begin{equation}
    \label{eq:suffDecSimp}
    \frac{1}{2}\eta(\p^{(k)})^\top\A\p^{(k)} \leq -(1 - \alpha)(\A\x^{(k)} + \b)^\top\p^{(k)}
\end{equation}
using the fact that \(\eta > 0\). The optimal step size computed for this problem will always be less than or equal to the unconstrained optimal step size 
\begin{equation*}
    \eta^{\star} \leq - \frac{(\A\x^{(k)} + \b)^\top\p^{(k)}}{(\p^{(k)})^\top\A\p^{(k)}}.
\end{equation*}
Plugging this into the left hand side of \cref{eq:suffDecSimp} we have 
\begin{align*}
    \frac{1}{2}\eta^\star(\p^{(k)})^\top\A\p^{(k)} &\leq -\frac{1}{2}(\A\x^{(k)}
 + \b)^\top\p^{(k)} \\
 &\leq -(1 - \alpha)(\A\x^{(k)}
 + \b)^\top\p^{(k)}
 \end{align*}
 as \(\p^{(k)}\) is a descent direction and \(\alpha \in (0, \frac{1}{2})\). Thus, the optimal step size satisfies the sufficient decrease condition; all necessary hypotheses of \cref{th:convergenceLee} are met, which implies that \cref{alg:pqn} is globally convergent.
\end{proof}

\section{Computing \texorpdfstring{\(c(\A)\)}{c(A)} and applying \texorpdfstring{\cref{lemma:cottleLemma}}{refcottleLemma}} \label{app:LipshitzConst}
\cottleLemma

 We see that as we increase \(\delta\), the Lipschitz constant grows. This suggests that matrices closer together will also have solutions closer together, as we expect. To see if we can apply the bound derived from \cref{lemma:cottleLemma}, this comes down to checking if \(\|\A - \A^\prime\|_{\infty} < c(\A)\). We now discuss how we can compute \(c(\A)\). Finding $c(\A)$ requires optimizing a nonconvex non-smooth objective. The non-convexity may not immediately obvious --- it comes from maximization occurring over quadratic forms that are not symmetric, 
\[h(\z) \defeq \max_{j \in \{1,\cdots,n\}} \z_j (\A\z)_j \neq \max_{j \in \{1,\cdots,n\}}  \bigl((\A^{1/2}\z)_j\bigr)^2.\]
In other words, if $\vec{e}_j$ is the $j^\text{th}$ unit vector, $h(\z) = \max_{j \in \{1,\cdots,n\}} \z^\top \left(\A\vec{e}_j\vec{e}_j^\top\right) \z$, and $\A\vec{e}_j\vec{e}_j^\top$ is not necessarily positive semidefinite.

Furthermore, the equality constraint is not amenable to standard optimization algorithms. This can be overcome by brute force:
\[\min_{\|\z\|_\infty = 1} h(\z) = \min_{\substack{i \in \{1,\cdots,n\} \\\sigma \in \{-1,1\}}} \min_{\substack{-1\leq \z \leq1 \\ z_i = \sigma} } h(\z). \]
Effectively, this approach loops over all possible $i$ and $\sigma$. For a fixed an index $i$ and $\sigma$, an equality constraint of $z_i = \sigma$ and inequality constraints $-1\leq\z\leq1$ are enforced. This guarantees that $\|\z\|_\infty =1$. Now local optimization can be performed on $h$. This is done with (semi-smooth) proximal gradient descent. The gradient is well defined almost everywhere
\[\nabla h(\z) = (\vec{e}_{j^*} \vec{e}_{j^*}^\top \A + \A^\top \vec{e}_{j^*} \vec{e}_{j^*}^\top) \z\]
where $j^* = \underset{j \in \{1,\cdots,n\}}{\operatorname{argmax}}\z_j (\A\z)_j$.

Cottle notes that $c(A)\in \bigl[ \tfrac{\lambda_{\mathrm{min}}(\A)}{n},\lambda_{\mathrm{min}}(\A)\bigr]$ \citep{CottlePangStone2009}. This motivates initializing with the projection of the eigenvector corresponding to the smallest eigenvalue
\[\z^{(0)}_{\sigma i} = \min\left(\max\left(\frac{\sigma \operatorname{sgn}\bigl((\v_{\mathrm{min}})_i\bigr) \v_{\mathrm{min}}}{\lvert (\v_{\mathrm{min}})_i\rvert},-1\right),1\right).\]
A backtracking line search was used to guarantee sufficient decrease at every iteration. 

This procedure does not guarantee that the bound is found exactly, but the bound found computed through this procedure is often tighter than $\lambda_{\mathrm{min}}(\A)$. \Cref{fig:app-cottle} shows the percentage of low fidelity matrices parameterized by \((p, \epsilon_{\mathrm{gmres}})\) where we are guaranteed that \(\hat{\A}\) is in a locally Lipschitz interval with respect to the LCP solution map.
\begin{figure}
    \centering
    \includegraphics[width=0.5\textwidth]{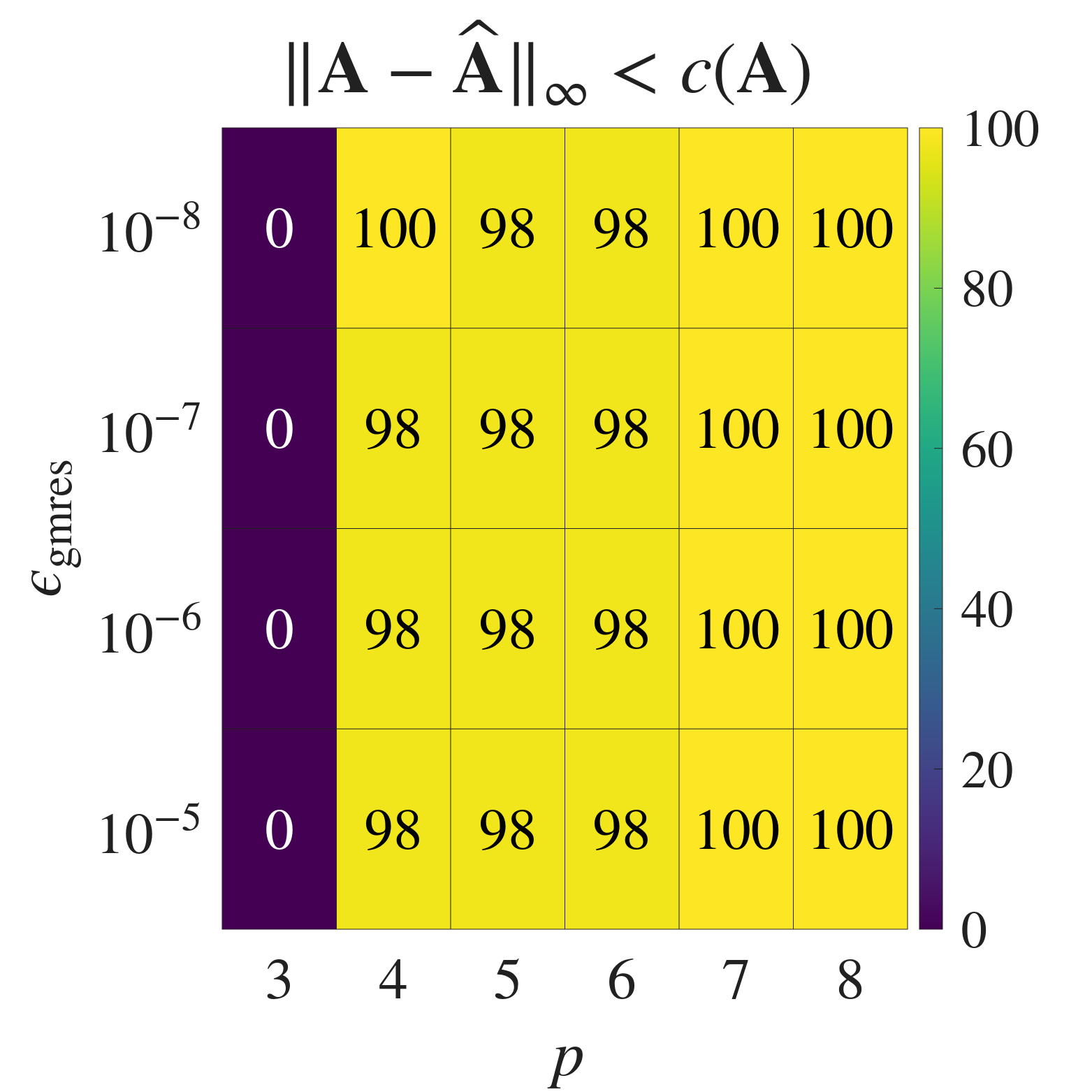}
    \caption{The percentage of the example LCPs generated in \cref{sec:res-offline} such that \(\left\|\A^\prime-\A^{\prime\prime}\right\|_{\infty} < c(\A)\) is true, as a function of spherical harmonic order \(p\) and GMRES tolernace $\epsilon_\text{gmres}$. Low fidelity matrices that satisfy this criterion have a locally Lipschitz solution map as given by \cref{lemma:cottleLemma}.}
    \label{fig:app-cottle}
\end{figure}
We see that this holds for nearly all examples except for \(p = 3\). While failing to meet this bound does not necessarily mean that the parameterization is a poor low fidelity model for \BPQN, our numerical results in \cref{sec:res-offline} showed that $p=4$ was a better choice than $p=3$. We present this to show that for many coarser discretizations we have guaranteed error bounds relating them to the high fidelity problem.

\section{Bifidelity Proximal Quasi-Newton} \label{app:bpqn}

\subsection{High Fidelity Secant Information}
The quasi-Newton update used to update the low fidelity matrix was the BFGS update, but other quasi-Newton updates are possible. The secant information corresponding to matrices of the form \(\mat{S}_{\text{high}}, \mat{Y}_{\text{high}}\) such that \(\A\mat{S}_{\text{high}} = \mat{Y}_{\text{high}}\) was not stored. Instead, the block updates were stored via unrolling as discussed in \cite{NocedalWright2006}. Other methods of storing the quasi-Newton updates are possible, such as densely or a compact representation \citep{ByrdNocedal1994MathProgram}. In limited memory settings, storing updates via unrolling is less computationally efficient than other competing approaches when memory falls out of the window. However, as the iteration counts were almost always less than or equal to 10, full memory was used, which avoids this issue. The computational cost of the BFGS update is the cost of computing
\begin{equation}
    \label{eq:BFGShigh}
    \B^{(k)}\s^{(k)} = \hat{\A} \s^{(k)} + \U^{(k)}(\U^{(k)})^T\s^{(k)} -  \V^{(k)}(\V^{(k)})^T\s^{(k)}
\end{equation}
plus \(\mathcal{O}(n)\). Thus, the total cost of applying $\B^{(k)}$ amounts to the cost of a low fidelity MVP plus \(\mathcal{O}(nk)\) where \(k\) is the current number of high fidelity secant conditions available; this is dominated by the cost of the low fidelity MVP. The low fidelity evaluation can be cached and reused which will be discussed in \cref{sec:bpqn:acrosssubs}. The total memory complexity is \(\mathcal{O}(nk)\).

\subsection{Across Subproblem Secant Conditions}
\label{sec:bpqn:acrosssubs}
When \cref{alg:pqn} terminates it does not update the secant conditions for that iteration as it is not of use. This is changed when \cref{alg:pqn} is called from \cref{alg:bpqn} as the secant conditions will be used in a future subproblem solve. In \cref{eq:BFGShigh}, an evaluation of \(\hat{\A}\s^{(k)}\) is needed where \(\s^{(k)} = \x^{(k + 1)} - \xk\). As the initial starting point of the \cref{alg:pqn} for the subproblem is \(\x^{(k)}, \hat{\A}\x^{(k)}\) has been computed. Similarly, the ending point of \cref{alg:pqn} for the subproblem is a low fidelity evaluation of \(\hat{\x}^{(k)}\). Thus one can compute 
\begin{equation*}
    \hat{\A}\s^{(k)} = \eta\left( \hat{\A}\hat{\x}^{(k)} - \hat{\A}\x^{(k)}\right)
\end{equation*}
for the needed low fidelity evaluation. 

\section{Initialization Procedure} \label{app:initConf}

For our implementation, pairs of particles are considered in the LCP if their pairwise distance is within a specified tolerance; see Eq.~\eqref{eq:collisional-set}. The goal of our numerical experiments is to isolate the computational improvement of our collision resolution algorithms. Thus, it was necessary to create an initial configuration $\CC$ where the particles are not touching yet, but collisions are possible after just a few time steps. Recall, the set $\AA$ contains all the active contact pairs for a particular configuration; here we write $\AA(\gamma,\CC)$ to denote the set's explicit dependence on the particle configuration $\CC$ and the scaling factor $\gamma$. We create this initial configuration using Algorithm~\ref{algo:initProc}.

\begin{algorithm}[ht]
\begin{algorithmic}
\Procedure{InitializeConfiguration}{$m, \Delta \x, \varepsilon_\mathrm{x}, n_\mathrm{desired}, \varepsilon_\mathrm{desired}$} 
    \State \(\CC \gets \) Initialized on a lattice of size $m\times m \times m$ centered at the origin where the center points are spaced $\Delta \x$ apart
    \State \(\CC \gets \CC + \mathscr{E}\) where $\mathscr{E} \sim$ Uniform$\bigl(-\tfrac{\varepsilon_\mathrm{x}}{2}, \tfrac{\varepsilon_\mathrm{x}}{2}\bigr)$
    \State Set $\gamma^-=0, \gamma = 1, \gamma^+=1$
    \While{$\lvert\AA(\gamma^+,\CC)\rvert > 0$} 
        \State $\gamma^+\gets 2\times \gamma^+$
    \EndWhile
    \While{$n_\mathrm{desired} - \varepsilon_\mathrm{desired} > \lvert\AA(\gamma ,\CC)\rvert  $ or $\lvert\AA(\gamma ,\CC)\rvert > n_\mathrm{desired} +  \varepsilon_\mathrm{desired}$}
        \If{$n_\mathrm{desired} - \varepsilon_\mathrm{desired} > \lvert\AA(\gamma ,\CC)\rvert $} 
            \State $\gamma^+ \gets \gamma$ 
            \State $ \gamma \gets \gamma + \frac{\gamma^+ - \gamma}{2}$
        \Else  
            \State $\gamma^- \gets \gamma$ 
            \State $ \gamma \gets \gamma^- + \frac{\gamma - \gamma^-}{2}$
        \EndIf
    \EndWhile
    \State \Return \(\gamma, \CC\)
\EndProcedure
\end{algorithmic}
\caption{Initializing a configuration with a set number of potential collisions.} \label{algo:initProc}
\end{algorithm}

The key observation is that because the configuration is centered at the origin, then we defined the function $\lvert\AA(\cdot ,\CC)\rvert: \mathbb{R}^+ \rightarrow \mathbb{N}$ which computes the number of pairs to consider in the LCP; notice, $\lvert\AA(\cdot ,\CC)\rvert$ is monotonic, non-increasing, and bounded from below by zero and from above by ${m^3 \choose 2}$ . Thus, first one can increase $\gamma^+$ until one finds a value such that $\lvert\AA(\gamma^+ ,\CC)\rvert = 0$. Also, a priori it is obvious that $\lvert\AA(0 ,\CC)\rvert = {m^3 \choose 2} = \mathcal{O}(m^6)$. This corresponds to the nonphysical configuration where all the particles are on top of each other. \Cref{algo:initProc} performs a bisection search on the scalar $\gamma$ to find a configuration where the initial size of the LCP is in $[n_\mathrm{desired} - \varepsilon_\mathrm{desired}, n_\mathrm{desired} + \varepsilon_\mathrm{desired}]$.

We found good results by setting $\varepsilon_x = 3$, $n_\mathrm{desired} = \tfrac{m^3}{2}$, and $\varepsilon_\mathrm{desired} = \tfrac{m^3}{10}$.

\end{document}